\documentclass{siamart251216}

\makeatletter
\@ifundefined{MSCcodes}{%
  \newenvironment{MSCcodes}{\begin{AMS}}{\end{AMS}}%
}{}
\makeatother

\usepackage[numbers]{natbib}
\usepackage{amsmath,amssymb,amsfonts,mathtools}

\newsiamremark{remark}{Remark}
\newsiamremark{example}{Example}
\newsiamthm{assumption}{Assumption}

\newcommand{\R}{\mathbb{R}}

\newcommand{\Npos}{\mathbb{N}_{>0}}  
\newcommand{\norm}[1]{\left\lVert#1\right\rVert}

\title{Penalised and constrained geodesics \\ in geometric control theory\thanks{\textbf{Funding:} This work was partly funded by the European Union under the project ROBOPROX (reg.\ no.\ CZ.02.01.01/00/22\_008/0004590).}}

\author{Rufus Lawrence\thanks{Faculty of Electrical Engineering, Czech Technical University in Prague, Karlovo N\'{a}m\v{e}st\'{i} 13, Prague, 11\,000, Czech Republic (\email{lawrejam@fel.cvut.cz}).}
\and Ale\v{s} Wodecki\footnotemark[2]
\and Johannes Aspman\footnotemark[2]
\and Jakub Mare\v{c}ek\footnotemark[2]}

\headers{Penalised and constrained geodesics}{R. Lawrence, A. Wodecki, J. Aspman, and J. Mare\v{c}ek}

\begin{document}

\maketitle

\begin{abstract}
In many problems in optimal control, one seeks to minimise an objective function subject to constraints on the velocity of the system. The ``hard-constrained approach'' (imposing the constraints directly) is often analytically and computationally challenging. A natural alternative is to penalise violations of the constraints, solving a sequence of ``soft-constrained'' problems indexed by a penalty parameter, in the hope that their solutions converge to a solution of the hard-constrained problem as the penalty tends to infinity. We justify this approach for a broad class of geometric control problems on a complete Riemannian manifold, in which the velocity constraint is encoded by a bracket-generating subbundle of the tangent bundle. In the absence of drift, the control problem reduces to that of finding a curve of minimal length or energy between two points, and we show that every sequence of solutions to the soft-constrained problems admits a subsequence converging, strongly in the Sobolev topology, to a solution of the hard-constrained problem. Finally, we show how to transform a broad class of optimal control problems with drift into problems of this geodesic type, by trivialising the dynamics using a change of coordinates inspired by the interaction picture of quantum mechanics.
\end{abstract}



\section{Introduction}\label{sec:introduction}

Let $A$ denote a set, let $S \subset A$, and let $f: A \to \R$. Suppose we want to solve the minimisation problem
\begin{align}
\begin{split}
    \underset{x \in A}{\min}\; f(x) \\
    \text{s.t.}\;x \in S.
\end{split}
\end{align}
The obvious solution is simply to forget that $f$ was ever defined on $A$, and instead work with its restriction $f|_S:S \to \R$. This approach is referred to as the \textit{hard-constrained} problem, and is often impractical, particularly when $A$ and $f$ possess certain ``nice'' properties (e.g. affineness, convexity) which $S$ and $f|_S$ do not. Typically, one overcomes this problem by introducing a Lagrange multiplier in the objective function to enforce the constraints. 

An alternative approach is to add a term $\psi(x)$ to the objective function that penalises points $x \in A$ lying far away (in some appropriate sense) from $S$:
\begin{align}
    \underset{x \in A}{\min}\; f(x) + q\psi(x),
\end{align}
where $q > 0$ is a penalty parameter. Minimising the modified (penalised) objective function is referred to as the \textit{soft-constrained} problem. Soft constraints are ubiquitous in optimisation and machine learning. The most common examples include weight decay in the training of neural networks \citep{krogh1991simple} and physics-informed neural networks, which incorporate physical constraints via an additional term in the loss function \citep{RAISSI2019686}. Soft constraints have also been implemented in the context of collision avoidance in motion planning \citep{zucker2013chomp}, where rather than enforcing a strict distance constraint, ``repulsive fields'' are included in the dynamics which discourage but do not prohibit collision. 

A natural question to ask is under what circumstances the solutions of the soft-constrained problem converge to solutions of the hard-constrained problem as the penalty parameter $q$ tends to infinity. The answer to this question is often affirmative, but in general depends on the properties of $A,S,f$ and $\psi$. For the finite-dimensional case, the convergence is discussed in Chapter 4 of Bertsekas \cite{bertsekas1999nonlinear} as well as Chapter 17 of Nocedal and Wright \cite{nocedal2006numerical}. 

When both $A$ and $S$ are infinite dimensional, the question becomes more delicate. If there is a natural choice of topology on $A$, the concept of $\Gamma$-convergence, first introduced by De Giorgi \cite{DeGiorgi}, is a powerful tool (see Definition \ref{def:gammaconvergence}). If we define
\begin{align}
    \tilde{f}(x) = 
     \begin{cases} 
     f(x) \;\text{if}\; x \in S \\
      \infty \;\text{otherwise},
      \end{cases}
\end{align}
a sufficient condition for solutions of the soft-constrained problem to converge to solutions of the hard-constrained problem is that $f+q\psi \overset{\Gamma}{\to} \tilde{f}$ on $A$ as $q \to \infty$. This is a direct result of the \textit{fundamental theorem of $\Gamma$-convergence} (see Theorem \ref{thm:fundamentalgamma}). It should be noted that pointwise convergence of $f+q\psi$ to $\tilde{f}$ is not enough to guarantee that solutions of the soft-constrained problem converge to solutions of the hard-constrained problem. 

In this paper, we study the control of a first-order ordinary differential equation on a Riemannian manifold $(M,g)$, with autonomous dynamics generated by a vector field $X$ on $M$, and subject to a nonholonomic velocity constraint on the component of the velocity corresponding to the control, as well as initial and terminal equality constraints on the trajectory. The quantity we want to minimise is the integral of the kinetic energy imparted by the control field. The hard-constrained problem can be formulated as 
\begin{align}
\begin{split}
    &\underset{(Y, \gamma)}\min \int_0^1g(Y(t,\gamma(t)),Y(t, \gamma(t)))\,dt \\
    \text{subject to}\;&\dot{\gamma}(t) = X(t, \gamma(t)) + Y(t, \gamma(t))  \\
    &\gamma(0) = x,\, \gamma(1) = y,\, Y(t, \gamma(t)) \in D_{\gamma(t)}.
\end{split}
\end{align}
Here, $X$ is a globally defined vector field describing the intrinsic dynamics of the system, $Y$ is a vector field along $\gamma$ that we think of as a control field, and $D$ is a bracket-generating subbundle of $TM$ which encodes the constraint on the control. 

The natural formulation of the soft-constrained problem is obtained by ``deforming'' the metric $g$ to penalise motion perpendicular to the constraint bundle. Letting $P$ and $P^\perp$ denote the projections from $TM$ to $D$ and its orthogonal complement $D^\perp$ respectively, we define
\begin{equation}
    g_q(X,Y) = g(PX, PY) + qg(P^\perp X, P^\perp Y).
\end{equation}
We may now formulate the soft-constrained problem as
\begin{align}
\begin{split}
    &\underset{(Y, \gamma)}\min \int_0^1g_q(Y(t,\gamma(t)),Y(t, \gamma(t)))\,dt \\
    \text{subject to}\;&\dot{\gamma}(t) = X(t, \gamma(t)) + Y(t, \gamma(t))  \\
    &\gamma(0) = x,\, \gamma(1) = y.
\end{split}
\end{align}
It turns out that, in the appropriate (Sobolev) topology, accumulation points of solutions to the soft-constrained problem are solutions to the hard-constrained problem. This is shown in Theorem \ref{thm:maincontroltheorem}. 

We first begin by studying the case where $X$ is zero. The above optimisation problems thus reduce to the problem of finding a length-minimising curve from $x$ to $y$ in $M$. In the hard-constrained case, we require that $\dot{\gamma}(t) \in D_{\gamma(t)}$, and such a curve is a sub-Riemannian geodesic with respect to the sub-Riemannian structure induced by $g|_D$. In the soft-constrained case, the length-minimising curve is an unconstrained geodesic with respect to $g_q$. We show, using the techniques of $\Gamma$-convergence, that any convergent sequence of geodesics of $g_q$ converges (with respect to the Sobolev topology) to a sub-Riemannian geodesic with respect to $(D, g|_D)$. While this result is intuitively unsurprising, the proof is technical, and makes up the bulk of this paper. 

In section \ref{sec:vakanomic}, we show how, under certain natural symmetry assumptions on $X(t, \cdot)$, $g$ and $D$, the optimal control problem can be reduced to the case where $X \equiv 0$. We further show that even in the absence of these symmetry assumptions, solutions to the soft-constrained problem still converge to solutions to the hard-constrained problem, even though the problem no longer reduces to the problem of finding geodesics.  

\section{Technical overview}
In this section, we fix notation and recall the technical background needed for the proofs. 

\subsection{Setup}
Let $M$ denote a Riemannian manifold with metric $g$ on $TM$. We assume throughout that $(M,g)$ is connected and complete. Let $D$ be a bracket-generating subbundle of the tangent bundle $TM$. The metric $g$ induces a splitting $TM \cong D \oplus D^\perp$, where $D^\perp$ denotes the orthogonal complement of $D$. We denote by $P$ and $P^\perp$ the orthogonal projections onto $D$ and $D^\perp$ respectively. 

We introduce a family of metrics indexed by $q \in \Npos$ by setting
\begin{equation}\label{eq:gq}
    g_q(X,Y) = g(PX, PY) + qg(P^\perp X, P^\perp Y)
\end{equation}
for tangent vectors $X$ and $Y$. This rescales the lengths of tangent vectors in $D^\perp$ by a factor of $\sqrt{q}$. Heuristically, we should think of this as making movement in directions perpendicular to $D$ more costly. Since $g \leq g_q \leq q\,g$, the induced distances satisfy $d_g \leq d_q \leq \sqrt{q}\,d_g$, where $d_q$ denotes the distance induced by $g_q$; in particular each $(M,g_q)$ is again complete. 

\begin{remark}\label{rem:integerq}
We take the penalty parameter to range over the positive integers, $q \in \Npos = \{1,2,3,\dots\}$. This is a matter of convenience rather than substance: it makes the family $\{E_q\}$ introduced below a genuine sequence, so that the sequential language of $\Gamma$-convergence and of subsequence extraction in Section \ref{sec:convergence} applies verbatim. Nothing in what follows depends on integrality: every statement and proof remains valid for real $q \geq 1$, with $q \to \infty$ understood along an arbitrary sequence.
\end{remark}

Associated to $g_q$, there is an energy functional $E_q$ acting on continuously differentiable curves $\gamma:[0,1] \to M$, defined by
\begin{equation}
    E_q(\gamma) = \frac{1}{2}\int_0^1g_q(\dot{\gamma}(t), \dot{\gamma}(t))\,dt,
\end{equation}
and a length functional 
\begin{equation}
    l_q(\gamma)= \int_0^1\sqrt{g_q(\dot{\gamma}(t), \dot{\gamma}(t))}\,dt.
\end{equation}

Standard results in Riemannian geometry tell us that curves with fixed endpoints $x$ and $y$ minimising $E_q$ also minimise $l_q$, and that curves minimising $l_q$ will minimise $E_q$ upon constant speed reparametrisation. Since $(M,g_q)$ is complete, such minimisers exist for any pair of endpoints by the Hopf--Rinow theorem. In the convergence analysis of Section \ref{sec:convergence} we work throughout with the energy functional $E_q$; the length functional enters only as the geometric object whose minimisers are the sub-Riemannian geodesics. 

We may also define two ``limit'' functionals:
\begin{align}
      E_\infty(\gamma) =   
      \begin{cases} 
     \frac{1}{2}\int_0^1g(\dot{\gamma},\dot{\gamma})\,dt \; &\text{if}\; \dot{\gamma}(t) \in D_{\gamma(t)} \; \text{a.e.} \\
      \infty  &\text{otherwise}, \\
    \end{cases}
    \end{align}
\begin{align}
      l_\infty(\gamma) =   
      \begin{cases} 
     \int_0^1\sqrt{g(\dot{\gamma},\dot{\gamma})}\,dt \; &\text{if}\; \dot{\gamma}(t) \in D_{\gamma(t)} \; \text{a.e.} \\
      \infty  &\text{otherwise}. \\
    \end{cases}
    \end{align}

The functional $l_\infty$ is referred to as the Carnot--Carath\'{e}odory length of the path $\gamma$ in the context of sub-Riemannian geometry. Moreover, since $D$ is a bracket-generating distribution (i.e. the Lie closure of $D$ is the whole of $TM$), the Chow--Rashevskii theorem tells us that any two points $x$ and $y$ in $M$ can be connected by a curve $\gamma$ whose tangent lies in $D$ (curves satisfying this property are called \textit{horizontal}). This allows one to define the Carnot--Carath\'{e}odory distance,
\begin{align}
    d_\infty(x,y) = \underset{\gamma}{\inf}\,l_\infty(\gamma), 
\end{align}
where the infimum must now be taken over a larger class of curves of lower regularity: it is not sufficient to take this infimum over smooth (or even $C^1$) horizontal curves connecting $x$ to $y$. One approach is to work with \textit{admissible curves}, which are defined as Lipschitz curves whose derivative $\dot{\gamma}$ lies in $D$ for almost all $t$. This is the approach taken by Agrachev, Barilari, and Boscain \cite{Agrachev_Barilari_Boscain_2019}. An equivalent approach is to take the infimum over the Sobolev space $H^1([0,1],M)$ (see Definition \ref{def:sobolevcurves}), as in Lewis \cite{Lewis2020}, which is the approach we take in this paper. In either setting, $E_\infty$ is finite on a nonempty set of curves joining $x$ to $y$. 

It is a basic result in sub-Riemannian geometry that for all $x,y \in M$, $d_q(x,y) \to d_\infty(x,y)$ as $q \to \infty$ (see for instance \cite{ledonne}). However, extracting the length- or energy-minimising limit curves is harder. It is clear that the limit $\lim_{q\to\infty}g_q$ is undefined, hence the ``usual'' approach of passing to the geodesic equation will not work, since the coefficients of the geodesic equation are determined (in local coordinates) by the Christoffel symbols of the metric, which do not exist in the limit. However, we will show in the next section that for an appropriate notion of convergence, namely $\Gamma$-convergence, $E_\infty$ may be interpreted as the limit functional of $E_q$. 

\subsection{Summary of results}

In this paper, we prove the following results:

\begin{itemize}
    \item Extending $E_q$ and $E_\infty$ to the space $\mathcal{C}_{x,y}$ of continuous curves from $x$ to $y$ by setting them equal to infinity away from $H^1([0,1],M,x,y)$, we prove that $E_q \overset{\Gamma}{\to}E_\infty$ with respect to the supremum norm (Theorem \ref{thm:gammaC0}). 
    \item Letting $\{\gamma_q\}$ denote a sequence of energy-minimising Riemannian geodesics of $g_q$, we construct a sequence of geodesics $\{\gamma_{q_j}\}$, $q_j \to \infty$, which converges to a path $\gamma \in C^0([0,1],M)$ using the Arzel\`{a}--Ascoli theorem (Proposition \ref{thm:limitconstruct}). 
    \item By the fundamental theorem of $\Gamma$-convergence, it follows that this path $\gamma$ is in fact a minimiser of $E_\infty$ (Proposition \ref{thm:limitisminimiser}). 
    \item Extracting a further subsequence, we show that $\gamma_{q_j} \to \gamma$ in $H^1([0,1],M)$ (Proposition \ref{thm:limitweakdiff}).
    \item The passage to a subsequence cannot be dispensed with in general: if the constrained problem admits several minimisers, the penalised geodesics may accumulate at more than one of them. Nevertheless, we prove that for the full sequence $\{\gamma_q\}$, the distance from $\gamma_q$ to the set of minimisers of $E_\infty$ tends to zero in the $H^1$-topology. Moreover, the constraint violation $\int_0^1 g(P^\perp\dot{\gamma}_q, P^\perp\dot{\gamma}_q)\,dt$ vanishes at an explicit rate $O(q^{-1})$. If the minimiser is unique then no subsequence selection is needed at all (Propositions \ref{prop:fullfamily} and \ref{prop:aprioribound}). 
    \item In Section \ref{sec:vakanomic}, we show how to reduce a class of control problems with drift to the problem of computing a geodesic, under natural symmetry assumptions on the drift with respect to the metric $g$ and constraint bundle $D$. The main insight is to derive a geometric generalisation of the interaction picture transformation in quantum mechanics. Together with the previous results, this establishes Theorem \ref{thm:maincontroltheorem}. Finally, we sketch how to treat the case of arbitrary drift. 
\end{itemize}  

\subsection{Constrained variational calculus}\label{sec:constvar}

The standard approach to minimising integral functionals such as $E_q$ and $E_\infty$ is to use techniques from the calculus of variations to reduce the minimisation problem to the problem of solving a second-order ODE (the Euler--Lagrange equation) with Dirichlet boundary conditions. In the case of $E_q$ this is straightforward: the resulting Euler--Lagrange equations are the usual Riemannian geodesic equations for the metric $g_q$. Minimising $E_\infty$ is more subtle: there are two natural ways to impose the constraints in the variational principle, one referred to as the \textit{vakonomic} approach and the other as the \textit{nonholonomic} approach. 

Suppose $L : TM \rightarrow \R$ is $C^2$. Define the functional $F : C^2([0,1],M)\rightarrow \R$ by
\begin{equation}
    F(\gamma) = \int_0^1L(\gamma(t), \dot{\gamma}(t))\,dt.
\end{equation}
The standard approach to minimising (extremising) $F$ is to use a variational principle. In the unconstrained case, this is simple: we find $\gamma \in C^2([0,1],M)$ such that $\delta F|_\gamma(h) = 0$ on all perturbations $h$ of $\gamma$ fixing the endpoints. Imposing the velocity constraint $\dot{\gamma}(t) \in D$ is more subtle, and in fact there are two approaches. 

The \textit{vakonomic} approach is to first restrict $F$ to the space $C_{\mathrm{hor}}(M)$ of horizontal curves in $M$, and then impose the variational principle: we require that 
\begin{equation}
    \delta(F|_{C_{\mathrm{hor}}})|_\gamma(h) = 0
\end{equation}
for all infinitesimal variations $h$ of $\gamma$ in $C_\mathrm{hor}(M)$ fixing the endpoints. This approach is entirely variational in nature: once we restrict $F$ to $C_\mathrm{hor}(M)$, we no longer care that $F$ was ever defined on a larger space of curves. 

By contrast, the \textit{nonholonomic} approach is to first take the variational derivative of $F$ on the full space $C^2([0,1],M)$ and then require that 
\begin{equation}
    \delta F|_\gamma(h) = 0
\end{equation}
on \textit{infinitesimally horizontal} variations $h$ preserving the endpoints. While both approaches yield a solution $\gamma$ that is horizontal, in general, we obtain different dynamics. Heuristically, the nonholonomic variational principle yields the experimentally observed dynamics for a constrained mechanical system, while the vakonomic approach is appropriate for optimal control problems. In particular, the vakonomic variational principle is the correct one for computing normal sub-Riemannian geodesics. This corresponds to the case when $L(\gamma(t),\dot{\gamma}(t)) = \frac{1}{2}g(\dot{\gamma}(t), \dot{\gamma}(t))$, describing the dynamics of a free particle. For a more detailed explanation, see \cite{LewisMurray, Lewis2020}. 

This paper presents an alternative approach to imposing the vakonomic variational principle in the case of the free particle. Rather than imposing the constraints directly, one can instead solve the unconstrained problem with penalty factor $q$, and then take the limit as $q \to \infty$. 

\subsection{Related work}
The idea of penalising certain tangent directions is ubiquitous in control theory and mechanics. Notable examples are the \textit{complexity metrics} introduced by Nielsen and collaborators \citep{nielsen2006geometric, dowling2006} in the context of quantum control and quantum complexity theory. The same ideas have been applied to the control of noisy quantum systems in \cite{lawrence2025}. The limiting case as the penalty parameter $q$ tends to infinity has been studied by Shizume et al.\ \cite{shizume}, albeit using a different approach. 

In sub-Riemannian geometry, families of penalised metrics of the form \eqref{eq:gq} have been studied in the context of minimiser regularity. Sub-Riemannian length minimisers (geodesics) can fail to be smooth in general \citep{chitour2025, rossi2026}, even when both the bundle $D$ and the bilinear form $g$ are. By contrast, in the Riemannian setting, if the metric $g$ is smooth, then so are its geodesics; for a $C^k$ metric with $k \geq 2$, geodesics are of class $C^{k+1}$.

We also draw the reader's attention to the work of Maione et al. \cite{Maione1, Maione2} on the $\Gamma$-convergence of integral functionals depending on vector fields. Here, the authors study the convergence of minima of a family of integral functionals, working, as in this paper, in the Sobolev topology. 

It should be noted that some of the ideas presented in this paper potentially constitute a kind of mathematical ``folklore'' and may well be known to some members of the sub-Riemannian geometry community. Notably, the ideas of the proofs are similar in spirit to Section 8.9 of \cite{Agrachev_Barilari_Boscain_2019}, although in this work the authors consider families of sub-Riemannian geodesics, rather than families of Riemannian geodesics converging to a sub-Riemannian geodesic as the metric becomes singular. To the best of the authors' knowledge, these results are not published anywhere in their current form. In particular, the full statement of Theorem \ref{thm:penalisedconvergence} and its proof using $\Gamma$-convergence, as well as the results for the full sequence in Subsection \ref{sec:subsequenceselection} and the results of Section \ref{sec:vakanomic} are novel.

\begin{remark}Montgomery \cite{Montgomery2002} cautions against various false proofs of regularity/smoothness of sub-Riemannian geodesics involving families of penalised metrics of the form \eqref{eq:gq}. The idea is to take a sub-Riemannian length minimiser $\gamma$, and approximate it by a sequence of geodesics $\gamma_q$ of $g_q$. There are two problems with this approach: the first is that the convergence may not be ``nice'' enough to preserve smoothness, the second is that not every sub-Riemannian length minimiser can be obtained as a limit of Riemannian geodesics. Indeed, Agrachev and Sarychev \cite{agrachev_sarychev} note that sub-Riemannian length minimisers can be isolated in the space of curves, up to reparametrisation. Nevertheless, these caveats do not apply to our results: we do not claim that the limiting curve $\gamma$ of Proposition \ref{thm:limitconstruct} is smooth, nor do we claim that all sub-Riemannian length minimisers arise in this way. 
\end{remark}

\section{Convergence results}\label{sec:convergence} 

Before we turn to the more technical notion of $\Gamma$-convergence, it is instructive to observe that $E_\infty$ is the \textit{pointwise} limit of $E_q$, in the sense that for all $C^1$-curves $\gamma$, $E_q(\gamma) \to E_\infty(\gamma)$. To see this, note that if $\dot{\gamma}(t) \in D_{\gamma(t)}$ for all $t$, we have $E_q(\gamma) = E_\infty(\gamma)$ for all $q$. If instead $\dot{\gamma}(t_0) \notin D_{\gamma(t_0)}$ for some $t_0 \in [0,1]$, then $\dot{\gamma}(t) \notin D_{\gamma(t)}$ in some open neighbourhood of $t_0$, and hence $E_q(\gamma) \to \infty = E_\infty(\gamma)$ (Lemma \ref{lemma:monotone} below upgrades this observation to the extended functionals). Unfortunately, pointwise convergence of functionals does not interact nicely with minimisation: if $\{\gamma_q\}$ is a sequence of paths such that $\gamma_q$ is a minimiser of $E_q$, and $\gamma$ is an accumulation point of $\{\gamma_q\}$, it does not follow that $\gamma$ is a minimiser of $E_\infty$. If we wish to make this claim, we instead need to prove that $E_\infty$ is the $\Gamma$-limit of $E_q$. 

\subsection{$\Gamma$-convergence}\label{sec:gammadef}

The notion of $\Gamma$-convergence was first introduced by De Giorgi \cite{DeGiorgi}. For a more recent reference, see Braides \cite{Braides2006}. In its full generality, $\Gamma$-convergence is defined on an arbitrary topological space, via topological $\Gamma$-lower and $\Gamma$-upper limits --- suprema, over the neighbourhoods of each point, of the limits inferior and superior of local infima; we refer to Dal Maso \cite{dalMaso} for this definition. On first-countable spaces, the general definition is equivalent to the following sequential characterisation \citep{dalMaso}, which we adopt as our definition, since every space appearing in this paper is metrisable or pseudometrisable, and hence first countable.

\begin{definition}\label{def:gammaconvergence}
Let $X$ be a first-countable topological space (for example, a metric space). A sequence of functionals $F_n : X \to \R \cup \{+\infty\}$ \emph{$\Gamma$-converges} to $F : X \to \R \cup \{+\infty\}$, written $F_n \xrightarrow[]{\Gamma} F$, if for every $x \in X$:
\begin{enumerate}
    \item[\textup{(i)}] \textbf{(Limit-inferior inequality)} for every sequence $(x_n)_{n \geq 1} \subset X$ 
    with $x_n \to x$,
    \begin{equation}
        F(x) \leq \liminf_{n \to \infty} F_n(x_n);
    \end{equation}
    \item[\textup{(ii)}] \textbf{(Recovery sequence)} there exists a sequence 
    $(x_n)_{n \geq 1} \subset X$ with $x_n \to x$ such that
    \begin{equation}
        \limsup_{n \to \infty} F_n(x_n) \leq F(x).
    \end{equation}
\end{enumerate}
Given a sequence $\{F_q\}_{q \in \Npos}$, we say that $F_q$ $\Gamma$-converges to $F$ as $q \to \infty$ if $F_{q_j} \xrightarrow[]{\Gamma} F$ for every subsequence $q_j \to \infty$.
\end{definition}

$\Gamma$-convergence possesses the remarkable property that taking $\Gamma$-limits commutes with minimisation. Indeed, Braides ~\cite{Braides2006} points out that the definition of $\Gamma$-convergence was \textit{maieutically} reverse-engineered so as to have this property. 

\begin{theorem}[Fundamental theorem of $\Gamma$-convergence; \cite{DeGiorgi, dalMaso}]\label{thm:fundamentalgamma}
    Let $F_n$ be a sequence of functionals on a first-countable topological space $X$ with $F_n \xrightarrow[]{\Gamma} F$, and for each $n$ let $x_n$ be a minimiser of $F_n$. Then any accumulation point of the sequence $\{x_n\}$ is a minimiser of $F$.
\end{theorem}

\subsection{Sobolev spaces of curves}\label{sec:sobolev}

So far, we have considered the energy functional $E_q$ as a functional on the space of continuously differentiable curves \\ $C^1([0,1],M)$. In fact, it is more natural to consider $E_q$ as a functional on the Sobolev space of curves in $M$ with $L^2$-integrable weak derivatives, $H^1([0,1],M)$. 

\begin{definition}[\cite{Lewis2020}]\label{def:sobolevcurves}
    Let $M$ be a manifold and let $s \geq 1$ be an integer (only $s = 1$ plays a role in this paper). We denote by $H^s([0,1],M)$ the set of maps $\gamma : [0,1] \to M$ such that for all smooth functions $f:M \to \R$, $f \circ \gamma \in H^s([0,1],\R)$.
\begin{itemize}
    \item Denote by $H^s([0,1],M, x,y)$ the subset of $H^s([0,1],M)$ such that $\gamma(0) = x$ and $\gamma(1) = y$.
    \item Denote by $H^s([0,1],M, x,y,D)$ the subset of $H^s([0,1],M,x,y)$ such that \\$\dot{\gamma}(t) \in D_{\gamma(t)}$ for almost all $t$.
\end{itemize}
\end{definition}

Of course, unlike $H^s([0,1],\R)$, $H^s([0,1],M)$ and its variants are not vector spaces; nevertheless, each $f$ gives rise to a semimetric, and a family of semimetrics allows us to define a topology in a natural way, by analogy with the topology induced by a family of seminorms on a vector space. We follow the approach of Lewis \cite{Lewis2020} in introducing a family of semimetrics indexed by $f$ in order to equip $H^s([0,1],M)$ with an appropriate topology. By analogy with the more familiar concept of a seminorm, a semimetric denotes a distance function satisfying the triangle inequality and symmetry conditions, but not necessarily the condition that $d(x,y) = 0 \implies x=y$. 

\begin{definition}\label{def:semimetrictopology}
    For $f \in C^\infty(M, \R)$, $a \in \{0, 1, \dots,s\}$ and $\gamma_1, \gamma_2 \in H^s([0,1],M)$, we define the semimetric
    \begin{equation}
        \rho^s_{a,f}(\gamma_1, \gamma_2) = \left(\int_0^1|(f \circ \gamma_1)^{(a)}(t)  - (f \circ \gamma_2)^{(a)}(t)|^2 dt \right)^\frac{1}{2}.
    \end{equation}
    The topology on $H^s([0,1],M)$ is the coarsest topology such that $\rho^s_{a,f}$ is continuous for all $f \in C^\infty (M, \R)$ and for all $a \in \{0, 1, \dots,s\}$. 
\end{definition}
    
Another way to understand the topology on $H^s([0,1],M)$ is as the initial topology of the family of maps 
\begin{align}
\begin{split}
    \mathrm{ev}_f:H^s([0,1],M) &\to  H^s([0,1],\R) \\
    \gamma &\mapsto f \circ \gamma .
\end{split}
\end{align}
We have the following proposition.

\begin{proposition}[\cite{Lewis2020}, Section 3.2]\label{prop:topchar}
\begin{enumerate}
    \item The Whitney embedding theorem provides us with an embedding $\varphi : M \hookrightarrow \R^N$. The topology of $H^s([0,1],M)$ is induced by the finite family of semimetrics $\rho^s_{a, f^l}$, where $f^l = \pi_l \circ \varphi$ and $\pi_l$ denotes the $l^{\mathrm{th}}$ coordinate projection $\R^N \to \R$. 
    \item In particular, the topology is induced by the single semimetric $\max_{a,l}\rho^s_{a,f^l}$, so $H^s([0,1],M)$ is pseudometrisable and hence first countable. Consequently, Definition \ref{def:gammaconvergence} applies on $H^s([0,1],M)$, and a subset $A \subseteq H^s([0,1],M)$ is closed if and only if it contains the limits of its convergent sequences.
\end{enumerate}
\end{proposition}

Let us collect some straightforward observations about the Sobolev space of curves which we will make use of later. 

\begin{remark}\label{rem:whitneychart}
Define $H^s([a,b],M)$ in the obvious way. Then:
\begin{enumerate}
    \item It holds that $\gamma_j \to \gamma$ in $H^s([0,1],M)$ if and only if there exists a finite cover of $[0,1]$ by closed intervals $\{[a_k,b_k]\}_{k=1}^K$ such that $\gamma_j \to \gamma$ in $H^s([a_k,b_k],M)$ for all $k=1, \dots, K$.
    \item As before, let $\{f^l\}_{l=1}^N$ denote the coordinate projections of the Whitney embedding $\varphi: M \hookrightarrow \R^N$, i.e. $f^l = \pi_l \circ \varphi$. By the implicit function theorem, around any $x \in M$ there exists a subset of these maps $f^{l_1}, \dots, f^{l_n}$ such that $(f^{l_1}, \dots, f^{l_n})$ is a coordinate chart in a neighbourhood $U$ of $x$. If the images of $\gamma_j$ and $\gamma$ are contained in $U$, then it is sufficient to check convergence with respect to the semimetrics $\rho^s_{a,f^{l_m}}$ for $l_m \in \{l_1, \dots, l_n\}$.
    \item If $\gamma_j \to \gamma$ uniformly with respect to the distance function on $M$ induced by $g$, then there exist a partition $0 = a_0 < a_1 < \dots < a_K = 1$, charts $(U_k, \varphi_k)$ of the form constructed in (2), which may be chosen with $\overline{U}_k$ compact, and an index $j_0$, such that, writing $I_k := [a_{k-1},a_k]$, we have $\gamma(I_k) \subseteq U_k$ and $\gamma_j(I_k) \subseteq U_k$ for all $j \geq j_0$ and all $k$. To check convergence $\gamma_j \to \gamma$ in $H^1([0,1],M)$, it is sufficient to check convergence of $\{\gamma_j|_{I_k}\}_{j \geq j_0}$ in $H^1(I_k,M)$ for each $k$, which in turn can be verified using the semimetrics $\rho^s_{a,f_k^{l_m}}$ coming from the coordinate chart constructed from the Whitney embedding on $U_k$ via the implicit function theorem. 
\end{enumerate}
The same statements hold, mutatis mutandis, with $[0,1]$ replaced by any closed subinterval. 
\end{remark}

\subsection{$\Gamma$-convergence of the energy functionals}\label{sec:gammaenergy}

Unfortunately, it is difficult to apply the fundamental theorem of $\Gamma$-convergence on $H^1([0,1],M)$. Instead, we will consider $E_q$ and $E_\infty$ as functionals on the space of continuous curves from $x$ to $y$,
\begin{equation}
    \mathcal{C}_{x,y} := \{\gamma \in C^0([0,1],M)\;|\;\gamma(0) = x,\, \gamma(1) = y\},
\end{equation}
equipped with the metric of uniform convergence
\begin{equation}
    d_{C^0}(\sigma_1,\sigma_2) := \sup_{t \in [0,1]} d_g(\sigma_1(t),\sigma_2(t)),
\end{equation}
by extending them by infinity away from $H^1([0,1],M,x,y)$. Explicitly, we define 
\begin{align}
      E_q(\gamma) =   
      \begin{cases} 
     \frac{1}{2}\int_0^1g_q(\dot{\gamma},\dot{\gamma})\,dt \; &\text{if}\; \gamma \in H^1([0,1],M,x,y) \\
      \infty  &\text{otherwise},\\
    \end{cases}
\end{align}
and
\begin{align}
      E_\infty(\gamma) =   
      \begin{cases} 
     \frac{1}{2}\int_0^1g(\dot{\gamma},\dot{\gamma})\,dt \; &\text{if}\; \gamma \in H^1([0,1],M,x,y,D) \\
      \infty  &\text{otherwise}. \\
    \end{cases}
\end{align}
More generally, for a closed subinterval $I \subseteq [0,1]$ we write 
\begin{equation}
    E_q(\sigma; I) := \frac{1}{2}\int_I g_q(\dot{\sigma},\dot{\sigma})\,dt, \qquad \sigma \in H^1(I,M),
\end{equation}
extended to $C^0(I,M)$ by $+\infty$ away from $H^1(I,M)$, so that $E_q(\sigma) = E_q(\sigma;[0,1])$ for $\sigma \in \mathcal{C}_{x,y} \cap H^1([0,1],M)$. We retain the symbols $E_q$ and $E_\infty$ for the extended functionals; no confusion should arise. 

The remainder of this section is organised as follows. In the present subsection we establish lower semicontinuity of $E_q$ with respect to uniform convergence (Lemma \ref{lemma:lsc}) and monotone pointwise convergence of $E_q$ to $E_\infty$ (Lemma \ref{lemma:monotone}), and combine the two to prove $\Gamma$-convergence (Theorem \ref{thm:gammaC0}). Subsection \ref{sec:geodesicconv} contains the main result of the section, Theorem \ref{thm:penalisedconvergence}: minimisers of the penalised problems subconverge, strongly in $H^1([0,1],M)$, to a minimiser of the constrained problem. Subsection \ref{sec:subsequenceselection} records guarantees which hold for the full sequence of penalised geodesics, without passing to a subsequence. Subsection \ref{sec:lscgamma} examines the role played by $\Gamma$-convergence in the argument. 

\begin{lemma}\label{lemma:weakconv}
    Let $\{f_j\}_{j=1}^\infty \subset H^1([0,1], \R)$ satisfy $f_j \to f$ uniformly, and suppose $\dot{f}_j \rightharpoonup \eta$ weakly in $L^2([0,1],\R)$. Then $f \in H^1([0,1],\R)$ with weak derivative $\dot{f} = \eta$. 
\end{lemma}

\begin{proof}
    For any test function $\psi \in C_c^\infty((0,1))$, we have that 
    \begin{align}
\begin{split}
        \int_0^1\psi' f \,dt = \underset{j \to \infty}\lim\int_0^1\psi' f_j\, dt &=  -\underset{j \to \infty}\lim\int_0^1\psi \dot{f}_j \, dt \\ &= -\int_0^1\psi \eta\,dt,
    \end{split}
\end{align}
    using uniform convergence in the first equality and weak convergence in the last. Hence $\eta$ is the weak derivative of $f$.
\end{proof}

We can now establish lower semicontinuity of $E_q$ with respect to uniform convergence. We state the lemma over an arbitrary closed subinterval, as this is the form in which it will be reused in the proof of Proposition \ref{thm:limitweakdiff}.

\begin{lemma}[Lower semicontinuity of $E_q$ under uniform convergence]\label{lemma:lsc}
Fix $q \in \Npos$ and a closed subinterval $I \subseteq [0,1]$. If $\sigma_j \to \sigma$ uniformly on $I$, then
\begin{equation}
    E_q(\sigma; I) \;\leq\; \underset{j \to \infty}{\lim\inf}\, E_q(\sigma_j; I).
\end{equation}
\end{lemma}

\begin{proof}
Set $L := \liminf_{j} E_q(\sigma_j; I)$. If $L = \infty$ there is nothing to prove, so assume $L < \infty$ and pass to a subsequence (not relabelled) with $\lim_{j} E_q(\sigma_j; I) = L$ and $E_q(\sigma_j; I) \leq L+1$ for all $j$; in particular $\sigma_j \in H^1(I,M)$.

Since $\sigma_j \to \sigma$ uniformly, Remark~\ref{rem:whitneychart}(3) provides a partition $\min I = a_0 < a_1 < \dots < a_K = \max I$ of $I$, charts $(U_k, \varphi_k)$ with $\overline{U}_k$ compact, and an index $j_0$ such that $\sigma(I_k) \subseteq U_k$ and $\sigma_j(I_k) \subseteq U_k$ for all $j \geq j_0$, where $I_k := [a_{k-1},a_k]$. Fix $k$ and write $\xi_j := \varphi_k \circ \sigma_j|_{I_k}$ and $\xi := \varphi_k \circ \sigma|_{I_k}$, so that $\xi_j \to \xi$ uniformly on $I_k$. Let $(g_q)_{ab}$ denote the coefficients of $g_q$ in this chart. By compactness of $\overline{U}_k$, the metric is uniformly elliptic there: there exist $0 < \lambda \leq \Lambda$ with
\begin{equation}
    \lambda\,|v|^2 \;\leq\; (g_q)_{ab}(p)\,v^a v^b \;\leq\; \Lambda\,|v|^2 
\end{equation}
for all $p \in \overline{U}_k$ and $v \in \R^n$.

The energy bound and ellipticity give $\norm{\dot{\xi}_j}_{L^2(I_k)}^2 \leq \tfrac{2}{\lambda}(L+1)$. Since bounded sequences in the Hilbert space $L^2(I_k,\R^n)$ contain weakly convergent subsequences, we may extract a further subsequence (again not relabelled, and chosen simultaneously for the finitely many $k$) with $\dot{\xi}_j \rightharpoonup v$ weakly in $L^2(I_k,\R^n)$. Applying Lemma~\ref{lemma:weakconv} componentwise (using $\xi_j \to \xi$ uniformly), we obtain $\xi \in H^1(I_k,\R^n)$ with $\dot{\xi} = v$; in particular $\sigma \in H^1(I,M)$, so that the extension of $E_q(\cdot\,;I)$ by $+\infty$ is consistent: a uniform limit of curves with equibounded energy cannot leave $H^1(I,M)$.

Next, we define an inner product on $L^2(I_k,\R^n)$ whose coefficients are pulled back along the limit curve,
\begin{equation}
    \langle u, w \rangle_{L^2_{g_q}} := \int_{I_k} (g_q)_{ab}(\xi(t))\, u^a(t)\, w^b(t)\,dt,
\end{equation}
and denote by $L^2_{g_q}(I_k,\R^n)$ the resulting weighted space; its norm $\norm{\cdot}_{L^2_{g_q}}$ is equivalent to the standard $L^2$ norm by ellipticity. Then
\begin{align}
\begin{split}
    \int_{I_k} (g_q)_{ab}(\xi_j)\,\dot{\xi}_j^a \dot{\xi}_j^b \,dt \;=\; \norm{\dot{\xi}_j}_{L^2_{g_q}}^2 \\ +\; \int_{I_k} \bigl[(g_q)_{ab}(\xi_j) - (g_q)_{ab}(\xi)\bigr] \dot{\xi}_j^a \dot{\xi}_j^b \,dt,
\end{split}
\end{align}
and the second term is bounded in absolute value by 
\begin{equation}
    \sup_{t \in I_k}\bigl\Vert(g_q)_{ab}(\xi_j(t)) - (g_q)_{ab}(\xi(t))\bigr\Vert \cdot \tfrac{2}{\lambda}(L+1) \;\longrightarrow\; 0
\end{equation}
(in any matrix norm), since $(g_q)_{ab}$ is uniformly continuous on $\overline{U}_k$ and $\xi_j \to \xi$ uniformly. As norms are weakly lower semicontinuous,
\begin{align}
\begin{split}
    \norm{\dot{\xi}}_{L^2_{g_q}}^2 \;&\leq\; \underset{j\to\infty}{\lim\inf}\,\norm{\dot{\xi}_j}_{L^2_{g_q}}^2 \\ &=\; \underset{j\to\infty}{\lim\inf} \int_{I_k} (g_q)_{ab}(\xi_j)\,\dot{\xi}_j^a \dot{\xi}_j^b \,dt.
\end{split}
\end{align}

Summing over $k$ and using superadditivity of $\liminf$ (and suppressing the dependence of $\xi$ on $k$),
\begin{align}
\begin{split}
    E_q(\sigma; I) \;&=\; \frac{1}{2}\sum_{k=1}^{K} \norm{\dot{\xi}}_{L^2_{g_q}}^2 \\ &\leq\; \frac{1}{2}\sum_{k=1}^{K} \underset{j\to\infty}{\lim\inf} \int_{I_k} (g_q)_{ab}(\xi_j)\,\dot{\xi}_j^a \dot{\xi}_j^b \,dt \\ &\leq\; \underset{j\to\infty}{\lim\inf}\, E_q(\sigma_j; I) \;=\; L,
\end{split}
\end{align}
where the final bound is against the limit inferior of the original sequence, which justifies the initial extraction of a subsequence.
\end{proof}

\begin{lemma}[Monotone pointwise convergence]\label{lemma:monotone}
For every $\sigma \in \mathcal{C}_{x,y}$, the map $q \mapsto E_q(\sigma)$ is nondecreasing in $q \in \Npos$, and
\begin{equation}
    \sup_{q \in \Npos} E_q(\sigma) \;=\; \lim_{q \to \infty} E_q(\sigma) \;=\; E_\infty(\sigma).
\end{equation}
\end{lemma}

\begin{proof}
Monotonicity is immediate from $g_q \leq g_{q'}$ for $q \leq q'$ (and trivial off $H^1$). For the supremum, consider three cases. If $\sigma \notin H^1([0,1],M,x,y)$, then $E_q(\sigma) = \infty = E_\infty(\sigma)$ for all $q$. If $\sigma \in H^1$ is horizontal, then $P^\perp\dot{\sigma} = 0$ almost everywhere, so $E_q(\sigma) = \frac{1}{2}\int_0^1 g(P\dot{\sigma},P\dot{\sigma})\,dt = E_\infty(\sigma)$ for all $q$. If $\sigma \in H^1$ is not horizontal, then the integrands $\frac{1}{2}\bigl(g(P\dot{\sigma},P\dot{\sigma}) + q\,g(P^\perp\dot{\sigma},P^\perp\dot{\sigma})\bigr)$ increase pointwise in $q$ and tend to $+\infty$ on the set $\{t : P^\perp\dot{\sigma}(t) \neq 0\}$, which has positive measure; by the monotone convergence theorem $E_q(\sigma) \nearrow \infty = E_\infty(\sigma)$.
\end{proof}

\begin{theorem}[$\Gamma$-convergence on $\mathcal{C}_{x,y}$]\label{thm:gammaC0}
With respect to uniform convergence on $\mathcal{C}_{x,y}$,
\begin{equation}
    E_q \;\overset{\Gamma}{\longrightarrow}\; E_\infty \qquad \text{as } q \to \infty.
\end{equation}
\end{theorem}

\begin{proof}
Let $q_j \to \infty$. Since $(\mathcal{C}_{x,y},d_{C^0})$ is a metric space, it is in particular first countable, and it suffices to verify the two conditions of Definition~\ref{def:gammaconvergence} for the sequence $\{E_{q_j}\}$.

\emph{Limit-inferior inequality.} Let $\sigma_j \to \sigma$ uniformly. Fix $r \geq 1$. Since $q_j \to \infty$, we have $q_j \geq r$ for all but finitely many $j$, and for such $j$, $E_{q_j}(\sigma_j) \geq E_r(\sigma_j)$ by monotonicity (Lemma~\ref{lemma:monotone}); since the limit inferior is unaffected by finitely many terms, lower semicontinuity (Lemma~\ref{lemma:lsc} applied with the fixed parameter $r$) gives
\begin{equation}
    \underset{j\to\infty}{\lim\inf}\, E_{q_j}(\sigma_j) \;\geq\; \underset{j\to\infty}{\lim\inf}\, E_r(\sigma_j) \;\geq\; E_r(\sigma).
\end{equation}
As this holds for every $r \geq 1$, taking the supremum over $r$ and applying Lemma~\ref{lemma:monotone} yields
\begin{equation}
    \underset{j\to\infty}{\lim\inf}\, E_{q_j}(\sigma_j) \;\geq\; \sup_{r \geq 1} E_r(\sigma) \;=\; E_\infty(\sigma).
\end{equation}

\emph{Recovery sequence.} Take the constant sequence $\sigma_j = \sigma$. Then $\sigma_j \to \sigma$ uniformly, and by Lemma~\ref{lemma:monotone},
\begin{equation}
    \underset{j\to\infty}{\lim\sup}\, E_{q_j}(\sigma) \;=\; \lim_{q\to\infty} E_q(\sigma) \;=\; E_\infty(\sigma),
\end{equation}
in all three cases, including the non-horizontal and non-Sobolev ones, where both sides equal $+\infty$.
\end{proof}

\subsection{Convergence of penalised geodesics}\label{sec:geodesicconv}

We are now ready to state and prove, via three propositions, the main theorem of this section. 

\begin{theorem}[Convergence of penalised geodesics]\label{thm:penalisedconvergence}
Let $(M,g)$ be a connected Riemannian manifold which is complete as a metric space with respect to $d_g$. Let $D \subset TM$ be a bracket-generating subbundle, and fix $x, y \in M$. For each $q \in \Npos$, let $\gamma_q$ be a minimiser of $E_q$ over $\mathcal{C}_{x,y}$. Then there exist a subsequence $q_j \to \infty$ and a minimiser $\gamma$ of $E_\infty$ over $\mathcal{C}_{x,y}$ --- in particular, a horizontal curve of constant speed --- such that
\begin{enumerate}
    \item[\textup{(i)}] $\gamma_{q_j} \to \gamma$ in $H^1([0,1],M)$, and in particular uniformly;
    \item[\textup{(ii)}] $E_{q_j}(\gamma_{q_j}) \to E_\infty(\gamma) = \min E_\infty$ and $l_{q_j}(\gamma_{q_j}) \to l_\infty(\gamma)$;
    \item[\textup{(iii)}] $q_j \int_0^1 g\bigl(P^\perp\dot{\gamma}_{q_j}, P^\perp\dot{\gamma}_{q_j}\bigr)\,dt \to 0$.
\end{enumerate}
Moreover, $\min_{\mathcal{C}_{x,y}} E_q \nearrow \min_{\mathcal{C}_{x,y}} E_\infty$ as $q \to \infty$.
\end{theorem}

We begin by constructing the limiting curve as a uniform limit of a sequence of penalised geodesics.

\begin{proposition}\label{thm:limitconstruct}
    Let $(M,g)$, $D$, $x$, $y$ and $\{\gamma_q\}$ be as in Theorem \ref{thm:penalisedconvergence}. (Equivalently, each $\gamma_q$ is a constant-speed minimising geodesic of $g_q$ connecting $x$ and $y$; such a minimiser exists since $(M,g_q)$ is complete.) Then there exists a sequence $q_j \to \infty$ such that $\{\gamma_{q_j}\}$ converges uniformly to a limiting curve $\gamma \in \mathcal{C}_{x,y}$. 
\end{proposition}

\begin{proof}
Since $D$ is bracket-generating, the Chow--Rashevskii theorem provides a horizontal curve joining $x$ to $y$; reparametrising to constant speed, we obtain $\sigma \in H^1([0,1],M,x,y,D)$ with $C := E_\infty(\sigma) < \infty$. As $\sigma$ is horizontal, $E_q(\sigma) = E_\infty(\sigma) = C$ for every $q$. Since $\gamma_q$ minimises $E_q$ and $g \leq g_q$,
\begin{equation}
    \frac{1}{2}\int_0^1 g(\dot{\gamma}_q,\dot{\gamma}_q)\,dt \;=\; E_1(\gamma_q) \;\leq\; E_q(\gamma_q) \;\leq\; E_q(\sigma) \;=\; C,
\end{equation}
so that $\int_0^1 g(\dot{\gamma}_q,\dot{\gamma}_q)\,dt \leq 2C$. By the Cauchy--Schwarz inequality, for $0 \leq s \leq t \leq 1$,
\begin{align}
\begin{split}
    d_g(\gamma_q(s),\gamma_q(t)) \;&\leq\; \int_s^t \sqrt{g(\dot{\gamma}_q,\dot{\gamma}_q)}\,dr \\ &\leq\; \left(\int_s^t g(\dot{\gamma}_q,\dot{\gamma}_q)\,dr\right)^{1/2}|t-s|^{1/2} \\ &\leq\; (2C)^{1/2}|t-s|^{1/2},
\end{split}
\end{align}
so the sequence $\{\gamma_q\}$ is uniformly equicontinuous (indeed uniformly H\"older-$\tfrac12$). All images lie in the closed ball $\overline{B}_g(x,(2C)^{1/2})$, which is compact since $(M,g)$ is complete, by the Hopf--Rinow theorem. Hence, by the Arzel\`{a}--Ascoli theorem, there exists a sequence $q_j \to \infty$ such that $\{\gamma_{q_j}\}$ converges uniformly to a limiting path $\gamma \in C^0([0,1],M)$; since uniform convergence preserves the endpoints, $\gamma \in \mathcal{C}_{x,y}$. 
\end{proof}

\begin{proposition}\label{thm:limitisminimiser}
Let $\{\gamma_{q_j}\}$ and $\gamma$ be as in Proposition \ref{thm:limitconstruct}. Then $\gamma$ is a minimiser of $E_\infty$ over $\mathcal{C}_{x,y}$; in particular, $\gamma \in H^1([0,1],M,x,y,D)$ is horizontal, with $E_\infty(\gamma) = \min_{\mathcal{C}_{x,y}} E_\infty \leq C < \infty$. Moreover, $E_{q_j}(\gamma_{q_j}) \to E_\infty(\gamma)$.
\end{proposition}

\begin{proof}
By Theorem \ref{thm:gammaC0}, $E_{q_j} \xrightarrow[]{\Gamma} E_\infty$ on $(\mathcal{C}_{x,y}, d_{C^0})$. Since $\gamma_{q_j}$ minimises $E_{q_j}$ and $\gamma_{q_j} \to \gamma$ in this space, the fundamental theorem of $\Gamma$-convergence (Theorem \ref{thm:fundamentalgamma}) shows that $\gamma$ is a minimiser of $E_\infty$. Testing against the competitor $\sigma$ from the proof of Proposition \ref{thm:limitconstruct} gives $E_\infty(\gamma) \leq E_\infty(\sigma) = C < \infty$; in particular $\gamma \in H^1([0,1],M,x,y,D)$ is horizontal. 

For the convergence of energies, minimality and the horizontality of $\gamma$ give
\begin{equation}
    E_{q_j}(\gamma_{q_j}) \;\leq\; E_{q_j}(\gamma) \;=\; E_\infty(\gamma) \qquad \text{for every } j,
\end{equation}
while for each fixed $r \geq 1$, monotonicity (Lemma \ref{lemma:monotone}) and Lemma \ref{lemma:lsc} give 
\begin{equation}
    \liminf_j E_{q_j}(\gamma_{q_j}) \geq \liminf_j E_r(\gamma_{q_j}) \geq E_r(\gamma).
\end{equation} Taking the supremum over $r$ yields $\liminf_j E_{q_j}(\gamma_{q_j}) \geq E_\infty(\gamma)$.
\end{proof}

However, uniform convergence is not enough for control-theoretic purposes: we also wish to control the derivatives $\dot{\gamma}_{q_j}$. We therefore upgrade the convergence to strong convergence in $H^1([0,1],M)$.

\begin{proposition}\label{thm:limitweakdiff}
Let $\{\gamma_{q_j}\}$ and $\gamma$ be as in Proposition~\ref{thm:limitconstruct}. Then $\gamma_{q_j} \to \gamma$ in $H^1([0,1],M)$.  Moreover,
\begin{equation}
    E_1(\gamma_{q_j}) \to E_1(\gamma) \quad \text{and} \quad q_j\int_0^1 g\bigl(P^\perp\dot{\gamma}_{q_j}, P^\perp\dot{\gamma}_{q_j}\bigr)\,dt \to 0.
\end{equation}
\end{proposition}
 
\begin{proof}
The proof follows the scheme of the proof of Lemma~\ref{lemma:lsc}.We reduce to a computation in local coordinates, followed by a weak-compactness argument; the only new ingredient is the convergence of energies supplied by minimality, which upgrades the weak lower semicontinuity of norms used there to convergence of norms, and hence, via the Radon--Riesz property, to strong convergence.
 
Write $\mathcal{E} := E_\infty(\gamma) = \min E_\infty$; since $\gamma$ is horizontal, $E_1(\gamma) = \mathcal{E}$. By Proposition~\ref{thm:limitisminimiser}, $E_{q_j}(\gamma_{q_j}) \to \mathcal{E}$. Since $E_1(\gamma_{q_j}) \leq E_{q_j}(\gamma_{q_j})$, we have $\limsup_j E_1(\gamma_{q_j}) \leq \mathcal{E}$, while Lemma~\ref{lemma:lsc} (with $q = 1$) gives $\liminf_j E_1(\gamma_{q_j}) \geq E_1(\gamma) = \mathcal{E}$; hence $E_1(\gamma_{q_j}) \to E_1(\gamma)$. Consequently,
\begin{align}
\begin{split}
    E_{q_j}(\gamma_{q_j}) - E_1(\gamma_{q_j}) &= \frac{q_j - 1}{2}\int_0^1 g\bigl(P^\perp\dot{\gamma}_{q_j}, P^\perp\dot{\gamma}_{q_j}\bigr)\,dt \\ &\longrightarrow\; 0,
\end{split}
\end{align}
and since $q_j/(q_j - 1) \to 1$, the second claim follows.

For the $H^1$-convergence, we reduce to a computation in local coordinates. By uniform convergence and Remark~\ref{rem:whitneychart}(3), fix a partition $0 = a_0 < a_1 < \dots < a_K = 1$ with $I_m := [a_{m-1},a_m]$, charts $(U_m, \varphi_m)$ with $\overline{U}_m$ compact, and an index $j_0$ such that $\gamma(I_m) \subseteq U_m$ and $\gamma_{q_j}(I_m) \subseteq U_m$ for all $j \geq j_0$ and all $m$; we discard the finitely many indices $j < j_0$. Recall from Subsection \ref{sec:gammaenergy} the energies on subintervals, $E_1(\sigma; I_m) = \frac{1}{2}\int_{I_m} g(\dot{\sigma},\dot{\sigma})\,dt$, so that $E_1 = \sum_m E_1(\cdot\,;I_m)$.

Fix $k$. By the previous paragraphs the limit $\lim_j E_1(\gamma_{q_j})$ exists, so
\begin{align}
\begin{split}
    E_1(\gamma) &= \lim_{j} E_1(\gamma_{q_j})
    = \underset{j\to\infty}{\lim\sup} \sum_{m} E_1(\gamma_{q_j};I_m) \\
    &\geq \underset{j\to\infty}{\lim\sup}\, E_1(\gamma_{q_j};I_k) + \underset{j\to\infty}{\lim\inf} \sum_{m \neq k} E_1(\gamma_{q_j};I_m) \\
    &\geq \underset{j\to\infty}{\lim\sup}\, E_1(\gamma_{q_j};I_k) + \sum_{m \neq k} \underset{j\to\infty}{\lim\inf}\, E_1(\gamma_{q_j};I_m) \\
    &\geq \underset{j\to\infty}{\lim\sup}\, E_1(\gamma_{q_j};I_k) + \sum_{m \neq k} E_1(\gamma;I_m),
\end{split}
\end{align}
using, in order: the existence of the limit together with the decomposition $E_1 = \sum_m E_1(\cdot\,;I_m)$, the inequality $\limsup_j (a_j + c_j) \geq \limsup_j a_j + \liminf_j c_j$, the superadditivity of the limit inferior over the finite sum, and Lemma~\ref{lemma:lsc} on each $I_m$. Subtracting the finite quantity $\sum_{m \neq k} E_1(\gamma;I_m)$ from both sides and using $E_1(\gamma) = \sum_m E_1(\gamma;I_m)$ gives 
\begin{equation}
    \underset{j\to\infty}{\lim\sup}\, E_1(\gamma_{q_j};I_k) \;\leq\; E_1(\gamma;I_k);
\end{equation}
combined with Lemma~\ref{lemma:lsc} applied on $I_k$, we conclude that $E_1(\gamma_{q_j};I_k) \to E_1(\gamma;I_k)$ for every $k$.

Continuing with $k$ fixed, write $\xi_j := \varphi_k \circ \gamma_{q_j}|_{I_k}$ and $\xi := \varphi_k \circ \gamma|_{I_k}$, and retrace the proof of Lemma~\ref{lemma:lsc} with $q = 1$: the derivatives $\dot{\xi}_j$ are bounded in $L^2(I_k,\R^n)$, so every subsequence contains a further weakly convergent subsequence, whose limit is $\dot{\xi}$ by Lemma~\ref{lemma:weakconv}; hence $\dot{\xi}_j \rightharpoonup \dot{\xi}$ weakly in $L^2(I_k,\R^n)$. Moreover, the coefficient swap there shows $\norm{\dot{\xi}_j}_{L^2_g}^2 - 2E_1(\gamma_{q_j};I_k) \to 0$ in the pulled-back inner product $\langle u, w\rangle_{L^2_g} = \int_{I_k} g_{ab}(\xi(t))\,u^a w^b\,dt$. Where the proof of Lemma~\ref{lemma:lsc} could conclude only that $\norm{\dot{\xi}}_{L^2_g}^2 \leq \liminf_j \norm{\dot{\xi}_j}_{L^2_g}^2$, we now know $\norm{\dot{\xi}_j}_{L^2_g}^2 \to 2E_1(\gamma;I_k) = \norm{\dot{\xi}}_{L^2_g}^2$, and the identity
\begin{align}
\begin{split}
    \norm{\dot{\xi}_j - \dot{\xi}}_{L^2_g}^2 \;&=\; \norm{\dot{\xi}_j}_{L^2_g}^2 - 2\langle \dot{\xi}_j, \dot{\xi}\rangle_{L^2_g} + \norm{\dot{\xi}}_{L^2_g}^2 \\ &\longrightarrow\; \norm{\dot{\xi}}_{L^2_g}^2 - 2\norm{\dot{\xi}}_{L^2_g}^2 + \norm{\dot{\xi}}_{L^2_g}^2 \;=\; 0,
\end{split}
\end{align}
using weak convergence in the cross term and convergence of norms in the first, upgrades this to strong convergence (the Radon--Riesz property): $\dot{\xi}_j \to \dot{\xi}$ in $L^2(I_k,\R^n)$.

Finally, let $f \in C^\infty(M)$ be arbitrary. On each $I_k$, the chain rule gives $\frac{d}{dt}(f\circ\gamma_{q_j}) = \partial_a(f\circ\varphi_k^{-1})(\xi_j)\,\dot{\xi}_j^a$, a product of uniformly convergent coefficients and strongly $L^2$-convergent derivatives, so $\frac{d}{dt}(f\circ\gamma_{q_j}) \to \frac{d}{dt}(f\circ\gamma)$ in $L^2(I_k,\R)$. Summing over $k$ gives $\rho^1_{1,f}(\gamma_{q_j},\gamma) \to 0$, while $\rho^1_{0,f}(\gamma_{q_j},\gamma) \to 0$ follows from uniform convergence. As $f$ was arbitrary, $\gamma_{q_j} \to \gamma$ in $H^1([0,1],M)$.
\end{proof}

\begin{proof}[Proof of Theorem~\ref{thm:penalisedconvergence}]
Propositions \ref{thm:limitconstruct} and \ref{thm:limitisminimiser} produce a sequence $q_j \to \infty$ and a minimiser $\gamma$ of $E_\infty$ over $\mathcal{C}_{x,y}$ with $\gamma_{q_j} \to \gamma$ uniformly and $E_{q_j}(\gamma_{q_j}) \to E_\infty(\gamma) = \min E_\infty$; Proposition \ref{thm:limitweakdiff} upgrades the convergence to $H^1([0,1],M)$ and gives (iii).

It remains to prove the constant-speed and length assertions. Every horizontal curve $\sigma \in H^1([0,1],M,x,y,D)$ admits a constant-speed reparametrisation $\tilde{\sigma} \in H^1([0,1],M,x,y,D)$ with $l_\infty(\tilde{\sigma}) = l_\infty(\sigma)$ and $E_\infty(\tilde{\sigma}) = \frac{1}{2}l_\infty(\sigma)^2$ (see e.g. \cite[Section 3.6]{Agrachev_Barilari_Boscain_2019}), while the Cauchy--Schwarz inequality gives $2E_\infty(\sigma) \geq l_\infty(\sigma)^2$, with equality if and only if $\sigma$ has constant speed. Hence $\min E_\infty = \frac{1}{2}d_\infty(x,y)^2$. Since $l_\infty(\gamma) \geq d_\infty(x,y)$ by definition of $d_\infty$, while $l_\infty(\gamma)^2 \leq 2E_\infty(\gamma) = d_\infty(x,y)^2$ by the Cauchy--Schwarz inequality, equality holds throughout: $\gamma$ has constant speed and $l_\infty(\gamma) = d_\infty(x,y)$. Likewise, each $\gamma_{q_j}$, being a constant-speed minimising geodesic of $g_{q_j}$, satisfies $l_{q_j}(\gamma_{q_j})^2 = 2E_{q_j}(\gamma_{q_j})$, so by the convergence of energies, $l_{q_j}(\gamma_{q_j})^2 \to 2E_\infty(\gamma) = l_\infty(\gamma)^2$, proving (ii). (In particular $d_{q_j}(x,y) = l_{q_j}(\gamma_{q_j}) \to d_\infty(x,y)$, recovering the convergence of distances of \cite{ledonne}.)

For the final claim, $q \mapsto \min E_q$ is nondecreasing (Lemma \ref{lemma:monotone}) and bounded above by $\min E_\infty$, hence convergent as $q \to \infty$; by (ii) its limit along $\{q_j\}$ --- and therefore its limit --- is $\min E_\infty$.
\end{proof}

\subsection{Results for the full sequence}\label{sec:subsequenceselection}

The passage to a subsequence in Theorem~\ref{thm:penalisedconvergence} cannot be avoided in general: if $E_\infty$ admits several minimisers, the penalised geodesics may accumulate at more than one of them. Nevertheless, guarantees hold for the \emph{full} sequence $\{\gamma_q\}$, which is what matters if one computes $\gamma_q$ for a single large $q$. Both proofs below use the fact that the proof of Theorem~\ref{thm:penalisedconvergence} applies verbatim to any prescribed sequence $q_j \to \infty$. To measure distances in $H^1([0,1],M)$, let $\rho_{H^1} := \max\{\rho^1_{a,f^l} \,:\, a \in \{0,1\},\, l = 1,\dots,N\}$, where $\{f^l\}$ is the Whitney family of Proposition~\ref{prop:topchar}; then $\rho_{H^1}$ induces the topology of $H^1([0,1],M)$, and it separates points since $\varphi$ is an embedding.

\begin{proposition}[Convergence of the full sequence to the minimiser set]\label{prop:fullfamily}
Let the setting and the family $\{\gamma_q\}$ be as in Theorem~\ref{thm:penalisedconvergence}. Then, as $q \to \infty$:
\begin{enumerate}
    \item[\textup{(i)}] $\mathrm{dist}_{\rho_{H^1}}\bigl(\gamma_q,\, \arg\min E_\infty\bigr) \to 0$;
    \item[\textup{(ii)}] $E_1(\gamma_q) \to \min E_\infty$;
    \item[\textup{(iii)}] if the minimiser $\gamma$ of $E_\infty$ is unique, then $\gamma_q \to \gamma$ in $H^1([0,1],M)$.
\end{enumerate}
\end{proposition}

\begin{proof}
(i) Suppose not. Then there exist $\varepsilon > 0$ and a sequence $q_j \to \infty$ with $\mathrm{dist}_{\rho_{H^1}}(\gamma_{q_j}, \arg\min E_\infty) \geq \varepsilon$ for all $j$. Applying Theorem~\ref{thm:penalisedconvergence} to this sequence produces a further subsequence converging in $H^1([0,1],M)$ to a minimiser of $E_\infty$; since $\rho_{H^1}$ induces the topology of $H^1([0,1],M)$, the distance to the minimiser set tends to zero along this further subsequence, a contradiction.

(ii) Since $\gamma_q$ minimises $E_q$, testing against any horizontal minimiser gives $E_1(\gamma_q) \leq E_q(\gamma_q) = \min E_q \leq \min E_\infty$. Conversely, every sequence $q_j \to \infty$ admits, by Theorem~\ref{thm:penalisedconvergence} and Proposition~\ref{thm:limitweakdiff}, a further subsequence along which $E_1(\gamma_{q_j}) \to E_1(\gamma')$ for some minimiser $\gamma'$ of $E_\infty$; as $\gamma'$ is horizontal, $E_1(\gamma') = E_\infty(\gamma') = \min E_\infty$. A family each of whose sequences admits a further subsequence along which it converges to $\min E_\infty$ converges to $\min E_\infty$.

(iii) By the same argument, every sequence $q_j \to \infty$ admits a further subsequence along which $\gamma_{q_j} \to \gamma$ in $H^1([0,1],M)$, uniqueness forcing every subsequential limit to equal $\gamma$; since the topology of $H^1([0,1],M)$ is pseudometrisable, this yields $\gamma_q \to \gamma$ in $H^1([0,1],M)$.
\end{proof}

Thus, for $q$ large, $\gamma_q$ is uniformly close to \emph{some} solution $\gamma$ of the constrained problem (in general depending on $q$), and, in the charts of Remark~\ref{rem:whitneychart}, $\dot{\gamma}_q$ is $L^2$-close to $\dot{\gamma}$ --- whether or not $\gamma_q$ belongs to a distinguished subsequence. The constraint violation, moreover, admits a fully quantitative bound.

\begin{proposition}[A priori bound on the constraint violation]\label{prop:aprioribound}
In the setting of Theorem~\ref{thm:penalisedconvergence}, for every integer $q \geq 2$,
\begin{equation}
    \int_0^1 g\bigl(P^\perp\dot{\gamma}_q, P^\perp\dot{\gamma}_q\bigr)\,dt \;\leq\; \frac{2\min E_\infty - d_g(x,y)^2}{q-1};
\end{equation}
in particular, the normal component of $\dot{\gamma}_q$ vanishes in $L^2$ at the explicit rate $O(q^{-1/2})$, with a constant depending only on the data $(x,y,g,D)$.
\end{proposition}

\begin{proof}
By the definition \eqref{eq:gq} of $g_q$, the left-hand side equals \\ $2\bigl(E_q(\gamma_q) - E_1(\gamma_q)\bigr)/(q-1)$. Since $\gamma_q$ minimises $E_q$, testing against any horizontal minimiser gives $E_q(\gamma_q) = \min E_q \leq \min E_\infty$, while the Cauchy--Schwarz inequality gives $E_1(\gamma_q) \geq \frac{1}{2}d_g(x,y)^2$.
\end{proof}

\begin{corollary}\label{cor:fullnormal}
In the setting of Theorem \ref{thm:penalisedconvergence}, for the full sequence of minimisers,
\begin{equation}
    q\int_0^1 g\bigl(P^\perp\dot{\gamma}_q, P^\perp\dot{\gamma}_q\bigr)\,dt \;\longrightarrow\; 0 \qquad \text{as } q \to \infty.
\end{equation}
\end{corollary}

\begin{proof}
By Proposition \ref{prop:fullfamily}(ii), $E_1(\gamma_q) \to \min E_\infty$, while $E_q(\gamma_q) = \min E_q \nearrow \min E_\infty$ by Theorem \ref{thm:penalisedconvergence}. Hence $(q-1)\int_0^1 g(P^\perp\dot{\gamma}_q,P^\perp\dot{\gamma}_q)\,dt = 2\bigl(E_q(\gamma_q) - E_1(\gamma_q)\bigr) \to 0$, and $q/(q-1) \to 1$.
\end{proof}

\subsection{The role of $\Gamma$-convergence}\label{sec:lscgamma}

A critical reader may ask whether the machinery of $\Gamma$-convergence deployed in this section is truly necessary. The short answer is no. $\Gamma$-convergence enters the proof of Theorem~\ref{thm:penalisedconvergence} exactly once: in Proposition~\ref{thm:limitisminimiser}, we invoke the fundamental theorem of $\Gamma$-convergence to conclude that the limit of the minimisers is a minimiser of the limit. Instead of appealing to the fundamental theorem of $\Gamma$-convergence, we can instead prove this directly using lower semicontinuity. 

Consider the sequence of minimisers $\gamma_{q_j}$, converging uniformly to $\gamma \in \mathcal{C}_{x,y}$ as constructed in Proposition ~\ref{thm:limitconstruct}. Fix $r \in \Npos$. Monotonicity (Lemma~\ref{lemma:monotone}) gives $E_{q_j}(\gamma_{q_j}) \geq E_r(\gamma_{q_j})$ once $q_j \geq r$, lower semicontinuity (Lemma~\ref{lemma:lsc}) gives $\liminf_j E_r(\gamma_{q_j}) \geq E_r(\gamma)$, and taking the supremum over $r$ yields
\begin{equation}\label{eq:directliminf}
    E_\infty(\gamma) \;=\; \sup_{r \in \Npos} E_r(\gamma) \;\leq\; \underset{j\to\infty}{\lim\inf}\, E_{q_j}(\gamma_{q_j}).
\end{equation}
On the other hand, for an arbitrary competitor $\eta \in \mathcal{C}_{x,y}$, minimality of $\gamma_{q_j}$ and monotonicity give $E_{q_j}(\gamma_{q_j}) \leq E_{q_j}(\eta) \leq E_\infty(\eta)$. Consequently,
\begin{equation}\label{eq:directsandwich}
\begin{split}
    E_\infty(\gamma) \;\leq\; \underset{j\to\infty}{\lim\inf}\, E_{q_j}(\gamma_{q_j}) \;&\leq\; \underset{j\to\infty}{\lim\sup}\, E_{q_j}(\gamma_{q_j}) \\ &\leq\; E_\infty(\eta).
\end{split}
\end{equation}
Since $\eta$ was arbitrary, $\gamma$ minimises $E_\infty$; taking $\eta$ to be the Chow--Rashevskii competitor of Proposition~\ref{thm:limitconstruct} shows that $E_\infty(\gamma) < \infty$, hence $\gamma$ is horizontal, and taking $\eta = \gamma$ recovers the convergence of energies with which the proof of Proposition~\ref{thm:limitweakdiff} begins.

However, this direct proof should not come as a surprise, nor should one conclude that invoking $\Gamma$-convergence is ``overkill". 

\begin{itemize}
    \item $\Gamma$-convergence of the energy functionals follows straightforwardly from lower semicontinuity (Lemma \ref{lemma:lsc}) and monotone pointwise convergence (Lemma \ref{lemma:monotone}). Indeed, the limit-inferior inequality in the definition of $\Gamma$-convergence is itself a lower semicontinuity condition, imposed not on a single functional but on the family $\{E_q\}$ jointly in the point and the index. The $\Gamma$-limit of a nondecreasing sequence of lower semicontinuous functionals is simply its pointwise supremum \citep{dalMaso} --- which Lemmas~\ref{lemma:lsc} and \ref{lemma:monotone} identify as $E_\infty$. 
    \item The fundamental theorem of $\Gamma$-convergence follows almost directly from the definition (see for instance \cite{dalMaso}). Indeed, Braides \cite{Braides2006} remarks that the definition of $\Gamma$-convergence can be understood as having been ``maieutically" reverse engineered precisely so that this theorem holds. 
    \item The genuinely difficult parts of the proof of Theorem~\ref{thm:penalisedconvergence} do not involve $\Gamma$-convergence at all. The bulk of Section \ref{sec:convergence} is dedicated to extracting a uniformly convergent subsequence using the Arzel\`{a}--Ascoli theorem, and upgrading uniform convergence to $H^1$-convergence once minimality has been established. 
\end{itemize}
Thus, while $\Gamma$-convergence is not indispensable to this work, it nonetheless provides a convenient and elegant framework in which to present our results. 

\section{Vakonomic control theory}\label{sec:vakanomic}

In the final section of this paper, we show how a natural class of optimal control problems can be reduced to the problem of finding a geodesic. These optimal control problems can be naturally formulated in terms of either hard or soft constraints on a component of the velocity (specifically, the component corresponding to the control field). By reformulating the hard- and soft-constrained problems as the problems of finding sub-Riemannian and penalised Riemannian geodesics respectively, we are able to show, by invoking the results of Section \ref{sec:convergence}, that under certain symmetry assumptions the optimal controls of the soft-constrained problem converge to those of the hard-constrained problem in $L^2$; the associated trajectories converge in $H^1([0,1],M)$. The main tool we use is a generalisation of the interaction picture transformation of quantum mechanics. We present an application of the results of this paper to the problem of unitary
gate synthesis in quantum optimal control. Finally, we sketch the case where the drift term does not satisfy any symmetry assumptions. 

\subsection{The geometric interaction picture} 
Many problems in geometric control theory can be reduced to the problem of finding a sub-Riemannian geodesic. Consider the minimisation problem
\begin{align}\label{eq:driftoptimalmanifold}
\begin{split}
    &\underset{(Y, \gamma)}\min \int_0^1g(Y(t),Y(t))\,dt \\
    \text{subject to}\;&\dot{\gamma}(t) = X(t, \gamma(t)) + Y(t)  \\
    &\gamma(0) = x, \; \gamma(1) = y.
\end{split}
\end{align}

Here, $X$ is a time-dependent vector field on $M$, i.e. a smooth map $X : \R \times M \to TM$ with $X(t,\cdot) \in \Gamma(M,TM)$ for each $t$, describing the intrinsic dynamics of the system, and $Y$ is a vector field along $\gamma$ that we think of as a control field. For a subbundle $F \subseteq TM$, we write $\Gamma(\gamma,F)$ for the space of square-integrable vector fields along $\gamma$ with values in $F$, i.e. measurable maps $t \mapsto Y(t) \in F_{\gamma(t)}$ with $\int_0^1 g(Y,Y)\,dt < \infty$. One may consider the unconstrained case $Y \in \Gamma(\gamma,TM)$ or the constrained case $Y \in \Gamma(\gamma,D)$, where $D$ is a bracket-generating subbundle. Minimisation is over pairs $(Y,\gamma)$ with $\gamma \in H^1([0,1],M)$ and the ODE holding for almost every $t$; we call such pairs \emph{admissible}. Clearly, if $X = 0$ we recover the usual problem of finding a (constrained or unconstrained) geodesic connecting $x$ and $y$. 

On the other hand, if $X$ is non-zero, we should think of the objective function in \eqref{eq:driftoptimalmanifold} as describing the energy added to the system by the control, a natural quantity to minimise. For instance, in the context of controlling the Schr\"odinger equation, this problem was studied in \cite{lawrence2025}. Despite the presence of non-trivial intrinsic dynamics, it is possible to reformulate this problem as an energy-minimisation problem of the type studied in Section \ref{sec:convergence}. This change of coordinates is a generalisation of the more familiar \textit{interaction picture} of quantum mechanics. 

\begin{assumption}\label{ass:flow}
The vector field $X$ is smooth, and its evolution operator is globally defined on the time interval of the control problem: for every $s \in [0,1]$ and $p \in M$, the solution $\varphi_{t,s}(p)$ of 
\begin{equation}
    \frac{d}{dt}\varphi_{t,s}(p) = X(t, \varphi_{t,s}(p)), \qquad \varphi_{s,s}(p) = p,
\end{equation}
exists for all $t \in [0,1]$. Writing $\varphi_t := \varphi_{t,0}$, each $\varphi_t : M \to M$ is then a diffeomorphism, with inverse $\varphi_{0,t}$.
\end{assumption}

Existence from arbitrary initial times, rather than from $t = 0$ alone, is what makes each $\varphi_t$ invertible; forward existence by itself does not (consider $X(x) = -e^x$ on $\R$, whose time-$t$ map is a diffeomorphism onto a proper subset).

\begin{proposition}\label{prop:interaction}
Let $\varphi_t$ be the evolution maps of Assumption \ref{ass:flow}. For an admissible pair $(Y,\gamma)$ for \eqref{eq:driftoptimalmanifold}, define
\begin{equation}\label{eq:zetadef}
    \zeta(t) := \varphi_t^{-1}(\gamma(t)), \qquad \tilde{Y}(t) := (d\varphi_t^{-1})_{\gamma(t)}\, Y(t) \;\in\; T_{\zeta(t)}M.
\end{equation}
Then $\zeta \in H^1([0,1],M)$ with $\zeta(0) = x$ and $\zeta(1) = \varphi_1^{-1}(y)$, and
\begin{equation}\label{eq:zetadotODE}
    \dot{\zeta}(t) = \tilde{Y}(t) \qquad \text{for almost every } t.
\end{equation}
Conversely, given $\zeta \in H^1([0,1],M)$ with these endpoints, the pair defined by $\gamma(t) := \varphi_t(\zeta(t))$ and $Y(t) := (d\varphi_t)_{\zeta(t)}\,\dot{\zeta}(t)$ is admissible for \eqref{eq:driftoptimalmanifold}. Moreover, for almost every $t$ and every $q \in \Npos$,
\begin{align}\label{eq:costidentity}\begin{split}
    g_{\gamma(t)}\bigl(Y(t),Y(t)\bigr) &= (\varphi_t^*g)_{\zeta(t)}\bigl(\tilde{Y}(t),\tilde{Y}(t)\bigr), \\
    (g_q)_{\gamma(t)}\bigl(Y(t),Y(t)\bigr) &= \bigl(\varphi_t^*(g_q)\bigr)_{\zeta(t)}\bigl(\tilde{Y}(t),\tilde{Y}(t)\bigr).
\end{split}
\end{align}
\end{proposition}

\begin{proof}
The map $(t,p) \mapsto \varphi_t^{-1}(p) = \varphi_{0,t}(p)$ is smooth, so $\zeta \in H^1([0,1],M)$, and the endpoint values follow from $\varphi_0 = \mathrm{id}$ and $\gamma(1) = y$. Writing $\gamma(t) = \varphi_t(\zeta(t))$ and taking the total derivative gives, for almost every $t$,
\begin{equation}
     \dot{\gamma}
  = \frac{\partial \varphi_t}{\partial t}\bigg|_{\zeta}
    + (d\varphi_t)_\zeta\, \dot{\zeta}
  = X(t, \varphi_t(\zeta)) + (d\varphi_t)_\zeta\, \dot{\zeta};
\end{equation}
comparing with $\dot{\gamma} = X(t,\gamma) + Y(t)$ yields $(d\varphi_t)_\zeta\,\dot{\zeta} = Y(t)$, which is \eqref{eq:zetadotODE}. The converse is the same computation read in reverse. The identities \eqref{eq:costidentity} are the definition of the pullback metric applied to $\tilde{Y}(t) = (d\varphi_t^{-1})_{\gamma(t)}Y(t)$.
\end{proof}

\subsection{Transforming the metric} The next step is to reformulate the optimisation problem \eqref{eq:driftoptimalmanifold} in terms of $\zeta$ and $\tilde{Y}$. We begin by introducing the time-dependent pullback metric $\tilde{g}_t := \varphi_t^* g$; that is, for $p \in M$ and $V, W \in T_pM$,
\begin{equation}
    (\varphi_t^*g)_p(V,W) \;=\; g_{\varphi_t(p)}\bigl((d\varphi_t)_p V,\, (d\varphi_t)_p W\bigr).
\end{equation}
By \eqref{eq:costidentity}, the costs agree along corresponding pairs. Similarly, introducing the time-dependent subbundle $\tilde{D}_t := (d\varphi_t)^{-1}(D)$, we have $Y(t) \in D_{\gamma(t)}$ if and only if $\dot{\zeta}(t) \in (\tilde{D}_t)_{\zeta(t)}$; indeed, $(d\varphi_t)_p$ maps $\bigl((\tilde{D}_t)_p,\, \tilde{g}_t|_p\bigr)$ isometrically onto $\bigl(D_{\varphi_t(p)},\, g|_{\varphi_t(p)}\bigr)$, and likewise for the $\tilde{g}_t$- and $g$-orthogonal complements, so that $\varphi_t^*(g_q)$ coincides with the penalisation $(\tilde{g}_t)_q$ of $\tilde{g}_t$ with respect to $\tilde{D}_t$. By Proposition \ref{prop:interaction}, the optimisation problem \eqref{eq:driftoptimalmanifold} is therefore equivalent to the curve problem
\begin{align}\label{eq:Ytildeopt}
\begin{split}
    &\underset{\zeta}{\min}\int_0^1\tilde{g}_t(\dot{\zeta}(t), \dot{\zeta}(t))\,dt \\
    \text{subject to}\;&\zeta \in H^1([0,1],M), \\
    &\zeta(0) = x, \; \zeta(1) = \varphi_1^{-1}(y),
\end{split}
\end{align}
with the additional constraint $\dot{\zeta}(t) \in (\tilde{D}_t)_{\zeta(t)}$ for almost every $t$ in the constrained case; the control has been eliminated via $\tilde{Y} = \dot{\zeta}$. 

Due to the time dependence of $\tilde{g}_t$ and $\tilde{D}_t$, we have not yet reduced to the setting of Section \ref{sec:convergence}. Rather than treating the general case, which requires lifting the dynamics to the augmented manifold $M \times \R$ and constraining the additional coordinate, at some technical cost, we instead posit some additional assumptions on drift under which the time dependence can be removed entirely. The general case will be treated in a separate work. 

\begin{assumption}\label{ass:symmetry}
For every $t \in [0,1]$, writing $X_t := X(t,\cdot)$:
\begin{enumerate}
    \item[\textup{(i)}] $[X_t, \Gamma(D)] \subseteq \Gamma(D)$;
    \item[\textup{(ii)}] $\mathcal{L}_{X_t}\, g = \mu(t)\, g$ for a continuous function $\mu$ of time alone.
\end{enumerate}
\end{assumption}

Condition (i) states that the drift is an infinitesimal symmetry of the constraint bundle; condition (ii) that it is an infinitesimal homothety of the metric, with spatially constant rate. Both conditions are infinitesimal, and hence verifiable without integrating the flow. Given Assumption \ref{ass:symmetry}, define
\begin{equation}\label{eq:conformalclock}
    c(t)^2 := \exp\!\int_0^t \!\mu(r)\,dr, \qquad \tau(t) := \int_0^t \!c(r)^{-2}dr,
\end{equation}
together with $T := \tau(1)$ and the normalised clock $\hat{u}(t) := \tau(t)/T$, a $C^1$ diffeomorphism of $[0,1]$ with derivative bounded away from $0$ and $\infty$. 

\begin{lemma}\label{lemma:conformalstructure}
Under Assumptions \ref{ass:flow} and \ref{ass:symmetry}, for every $t \in [0,1]$ we have $\tilde{g}_t = c(t)^2\, g$ and $\tilde{D}_t = D$; moreover, the flow preserves $D^\perp$, so that for every $q \in \Npos$,
\begin{equation}
    (\tilde{g}_t)_q \;=\; c(t)^2\, g_q,
\end{equation}
and horizontality with respect to $\tilde{D}_t$ coincides with horizontality with respect to $D$.
\end{lemma}

\begin{proof}
For a time-dependent flow, $\frac{d}{dt}\varphi_t^*g = \varphi_t^*\mathcal{L}_{X_t}g = \mu(t)\,\varphi_t^*g$, whence $\varphi_t^*g = c(t)^2g$. For the bundle, let $\{V_i\}$ be a local frame of $D$ and write $[X_t, V_i] = \sum_j a_i^j(t,\cdot)\,V_j$ using (i). The pullbacks satisfy the linear system $\frac{d}{dt}(\varphi_t^*V_i) = \varphi_t^*[X_t,V_i] = \sum_j (a_i^j \circ \varphi_t)\,\varphi_t^*V_j$, whose solution with initial datum $V_i$ remains in the $C^\infty$-span of $\{V_j\}$, giving $(d\varphi_t)^{-1}(D) \subseteq D$; since each $(d\varphi_t)_p$ is an isomorphism, equality follows by comparing dimensions. Finally, a homothety preserving $D$ preserves $D^\perp$ and intertwines the orthogonal projections, whence $\varphi_t^*(g_q) = (\varphi_t^*g)_q = c(t)^2 g_q$.
\end{proof}

\subsection{Convergence of optimal controls for symmetric drift} We can now state the main results of this section. Throughout, let $(M,g)$ be a connected, complete Riemannian manifold, $D \subset TM$ a bracket-generating subbundle, and $X$ a time-dependent vector field satisfying Assumptions \ref{ass:flow} and \ref{ass:symmetry}; fix $x, y \in M$ and set $y' := \varphi_1^{-1}(y)$. Consider the hard-constrained problem
\begin{align}\label{eq:hardcontrol}
\begin{split}
    &\underset{(Y, \gamma)}\min \int_0^1g(Y(t),Y(t))\,dt \\
    \text{subject to}\;&\dot{\gamma}(t) = X(t, \gamma(t)) + Y(t)  \\
    &\gamma(0) = x, \; \gamma(1) = y, \; Y \in \Gamma(\gamma,D),
\end{split}
\end{align}
and the family of soft-constrained problems, indexed by $q \in \Npos$:
\begin{align}\label{eq:softcontrol}
\begin{split}
    &\underset{(Y, \gamma)}\min \int_0^1g_q(Y(t),Y(t))\,dt \\
    \text{subject to}\;&\dot{\gamma}(t) = X(t, \gamma(t)) + Y(t)  \\
    &\gamma(0) = x, \; \gamma(1) = y, \; Y \in \Gamma(\gamma,TM).
\end{split}
\end{align}

\begin{proposition}[Geodesic correspondence]\label{prop:geodesiccorrespondence}
The correspondence
\begin{equation}\label{eq:geodesiccorrespondence}
    \gamma(t) = \varphi_t\bigl(\xi(\hat{u}(t))\bigr), \qquad Y(t) = \dot{\gamma}(t) - X(t,\gamma(t))
\end{equation}
is a bijection between admissible pairs $(Y,\gamma)$ for \eqref{eq:softcontrol} and curves\\ $\xi \in \mathcal{C}_{x,y'} \cap H^1([0,1],M)$, under which
\begin{equation}
    \int_0^1 g_q(Y,Y)\,dt \;=\; \frac{2}{T}\,E_q(\xi);
\end{equation}
moreover, $Y \in \Gamma(\gamma,D)$ if and only if $\xi$ is horizontal, in which case $\int_0^1 g(Y,Y)\,dt = \frac{2}{T}E_\infty(\xi)$.
\end{proposition}

\begin{proof}
By Proposition \ref{prop:interaction}, $(Y,\gamma)$ is admissible for \eqref{eq:softcontrol} if and only if $\zeta(t) := \varphi_t^{-1}(\gamma(t))$ lies in $H^1([0,1],M)$ with $\zeta(0) = x$, $\zeta(1) = y'$, and then, by the cost identity \eqref{eq:costidentity} and Lemma \ref{lemma:conformalstructure},
\begin{align}
\begin{split}
    \int_0^1 g_q(Y,Y)\,dt \;&=\; \int_0^1 (\tilde{g}_t)_q(\dot{\zeta},\dot{\zeta})\,dt \\ &=\; \int_0^1 c(t)^2\, g_q(\dot{\zeta},\dot{\zeta})\,dt.
\end{split}
\end{align}
Since $\hat{u}$ is a $C^1$ diffeomorphism of $[0,1]$ with $\dot{\hat{u}} = c^{-2}/T$ bounded away from $0$ and $\infty$, the substitution $\zeta(t) = \xi(\hat{u}(t))$ is a bijection of the stated curve classes; and with $\dot{\zeta} = \dot{\hat{u}}\,\xi'(\hat{u})$,
\begin{align}
\begin{split}
    c^2 g_q(\dot{\zeta},\dot{\zeta})\,dt \;&=\; c^2\,\frac{c^{-4}}{T^2}\,g_q(\xi',\xi')\big|_{\hat{u}(t)}\,dt \\ &=\; \frac{1}{T}\,g_q(\xi',\xi')\big|_{\hat{u}(t)}\,d\hat{u},
\end{split}
\end{align}
which integrates to $\frac{2}{T}E_q(\xi)$. Horizontality is preserved in both directions since $\tilde{D}_t = D$ and monotone reparametrisation does not affect the constraint; the hard-cost identity follows in the same way with $q$ replaced by the horizontal restriction.
\end{proof}

\begin{theorem}\label{thm:maincontroltheorem}
Under the standing assumptions of this section:
\begin{enumerate}
    \item[\textup{(i)}] for every $q \in \Npos$, the problem \eqref{eq:softcontrol} admits a minimiser $(Y_q,\gamma_q)$, corresponding to a minimising geodesic $\xi_q$ of $g_q$ from $x$ to $y'$;
    \item[\textup{(ii)}] there exist a subsequence $q_j \to \infty$ and a minimiser $\xi$ of $E_\infty$ over $\mathcal{C}_{x,y'}$ such that $\xi_{q_j} \to \xi$ in $H^1([0,1],M)$ and the pair $(Y,\gamma)$ obtained from $\xi$ via \eqref{eq:geodesiccorrespondence} solves the hard-constrained problem \eqref{eq:hardcontrol}, and $\gamma_{q_j} \to \gamma$ in $H^1([0,1],M)$;
    \item[\textup{(iii)}] the optimal value of \eqref{eq:softcontrol} equals $\frac{2}{T}\min_{\mathcal{C}_{x,y'}} E_q$ and increases, as $q \to \infty$, to the optimal value $\frac{1}{T}\,d_\infty(x,y')^2$ of \eqref{eq:hardcontrol};
    \item[\textup{(iv)}] with $Y_{q_j}$ and $Y$ the controls corresponding to $\xi_{q_j}$ and $\xi$ as in \textup{(i)--(ii)}, and $\varphi : M \hookrightarrow \R^N$ the Whitney embedding of Proposition \ref{prop:topchar}, the controls converge in $L^2$ along the embedding:
    \begin{equation}
        d\varphi_{\gamma_{q_j}(t)}\,Y_{q_j}(t) \;\longrightarrow\; d\varphi_{\gamma(t)}\,Y(t) \qquad \text{in } L^2([0,1],\R^N).
    \end{equation}
\end{enumerate}
\end{theorem}

\begin{proof}
Claim (i) follows from Proposition \ref{prop:geodesiccorrespondence} and the Hopf--Rinow theorem, since $(M,g_q)$ is complete (Section \ref{sec:convergence}). 

For (ii), Theorem \ref{thm:penalisedconvergence} applied on $(M,g,D)$ with endpoints $x, y'$ produces $q_j \to \infty$ and a minimiser $\xi$ of $E_\infty$ with $\xi_{q_j} \to \xi$ in $H^1([0,1],M)$. Composition with the fixed diffeomorphism $\hat{u}$ preserves $H^1$-convergence, so $\zeta_{q_j} := \xi_{q_j}\circ\hat{u} \to \zeta := \xi\circ\hat{u}$ in $H^1([0,1],M)$; and $\gamma_{q_j} = \varphi_t \circ \zeta_{q_j} \to \varphi_t \circ \zeta = \gamma$ in $H^1([0,1],M)$ by the smoothness of the map $(t,p)\mapsto\varphi_t(p)$, the chain rule, and the chart-local characterisation of convergence in Remark \ref{rem:whitneychart}. The pair $(Y,\gamma)$ solves \eqref{eq:hardcontrol} by the cost identity and the horizontality correspondence of Proposition \ref{prop:geodesiccorrespondence}. 

Claim (iii) follows from the identity for the optimal values, the monotone convergence $\min E_q \nearrow \min E_\infty$ of Theorem \ref{thm:penalisedconvergence}, and the identity $\min E_\infty = \frac{1}{2}d_\infty(x,y')^2$ established in its proof. Finally, for (iv), write $d\varphi(Y_{q_j}) = \frac{d}{dt}(\varphi\circ\gamma_{q_j}) - d\varphi_{\gamma_{q_j}}\bigl(X(t,\gamma_{q_j})\bigr)$: the first term converges to $\frac{d}{dt}(\varphi\circ\gamma)$ in $L^2$ by the definition of the $H^1$ topology, and the second converges uniformly to $d\varphi_{\gamma}\bigl(X(t,\gamma)\bigr)$, by the smoothness of $X$ and $\varphi$, the uniform convergence $\gamma_{q_j} \to \gamma$, and the compactness of a common set containing all images.
\end{proof}

\begin{remark}[Symmetric drift]\label{rem:killingdrift}
If $\mu \equiv 0$ --- that is, each $X_t$ is a Killing field of $g$ preserving $\Gamma(D)$ --- then $c \equiv 1$, $T = 1$ and $\hat{u}(t) = t$, and the correspondence \eqref{eq:geodesiccorrespondence} is simply $\gamma(t) = \varphi_t(\xi(t))$: the drift problem \emph{is} the problem of Section \ref{sec:convergence}, with the target rotated to $\varphi_1^{-1}(y)$ and no reparametrisation at all. This is, in fact, the operative case of Assumption \ref{ass:symmetry} away from flat space. 

If $\mu \not\equiv 0$, then $c(t) \neq 1$ for some $t$, so that $\varphi_t$ is a \emph{non-isometric} homothety of the complete Riemannian manifold $(M,g)$. It follows from \cite[Lemma 2]{IshiharaObata1955} that $(M,g)$ must be flat (i.e. the Riemann curvature tensor vanishes identically). Replacing $\varphi_t$ by its inverse if necessary, we may assume $\varphi_t$ to be a contracting homothety. Since a contracting homothety scales the length of the shortest noncontractible loop, the fundamental group must be trivial. By the Killing--Hopf theorem \cite[Chapter 12]{Lee2018}, $(M,g)$ is then isometric to the Euclidean space $\R^n$. 

On Euclidean space, every homothety is an affine similarity. Thus, on any complete manifold other than flat $\R^n$, condition (ii) forces $\mu \equiv 0$, and the conformal allowance in Assumption \ref{ass:symmetry} is a costless extension whose content is confined to dilational drifts on Euclidean space.
\end{remark}

\subsection{A practical example from quantum optimal control}\label{sec:quantumexample}

Consider the manifold $M = SU(2^n)$ with the bi-invariant metric $g$ induced by the negative of the Killing form. Up to positive scaling, this metric is induced by the $\mathrm{Ad}$-invariant inner product $\langle A, B\rangle := -\operatorname{tr}(AB)$ on $\mathfrak{g} := \mathfrak{su}(2^n).$

We recall the standard terminology. Write $\sigma_1, \sigma_2, \sigma_3$ for the Pauli matrices and $\sigma_0 = I$ for the $2\times 2$ identity. A \emph{Pauli string} on $n$ qubits is a tensor product
\begin{equation}
    P = \sigma_{j_1} \otimes \sigma_{j_2} \otimes \cdots \otimes \sigma_{j_n}, \qquad j_1, \dots, j_n \in \{0,1,2,3\},
\end{equation}
and its \emph{support} is the set $\mathrm{supp}(P) := \{ i : j_i \neq 0 \}$ of qubits on which it acts non-trivially. The $4^n$ Pauli strings form a basis of the Hermitian $2^n \times 2^n$ matrices, and a Hermitian operator is called \emph{$k$-local} if it is a linear combination of Pauli strings $P$ with $|\mathrm{supp}(P)| \leq k$. Thus a $1$-local operator is a sum of single-qubit terms and a $2$-local operator additionally allows pairwise interactions. It is a standard fact (see for instance \cite{d2021introduction}) that 2-local Hamiltonians allow one to generate any unitary evolution. 
\thanks{Note that $k$-locality constrains the number of qubits addressed by each term, not their spatial proximity. However, it is straightforward to modify the above argument to impose additionally the constraint that only spatially adjacent qubits can interact.}

Let $\mathfrak{d} \subset \mathfrak{g}$ be the real span of the elements $-iP$, where $P$ ranges over the Pauli strings with $|\mathrm{supp}(P)| \leq 2$: the \emph{2-local} directions. Consider the right-invariant control system
\begin{equation}
    \dot{U}(t) = \bigl(A_d + A_c(t)\bigr)U(t), \qquad A_c(t) \in \mathfrak{d},
\end{equation}
where $A_d = -iH_d$ and $A_c(t) = -iH_c(t)$ for Hermitian drift and control Hamiltonians $H_d$ and $H_c(t)$, the latter 2-local. In the notation of this section, $X(U) = A_dU$, $Y(t) = A_c(t)U(t)$, and $D_U = \mathfrak{d}\cdot U$ is a right-invariant subbundle, bracket-generating by the universality of two-qubit interactions; the cost
\begin{equation}
    \int_0^1 g(Y,Y)\,dt = \int_0^1 \langle A_c, A_c\rangle\,dt
\end{equation}
is the control energy, as in the complexity-geometry programme \citep{nielsen2006geometric, dowling2006} and in \cite{lawrence2025}.

This control problem satisfies the hypotheses of Theorem \ref{thm:maincontroltheorem} exactly, with each assumption expressing a structural feature of the physical system. Assumption \ref{ass:flow} holds because the defining ODE is linear, so the evolution operator is globally defined; the flow consists of the left translations $\varphi_t = L_{a_t}$, $a_t = e^{tA_d}$. Since left translations are isometries of a bi-invariant metric, Assumption \ref{ass:symmetry}(ii) holds with $\mu \equiv 0$: the drift is a Killing field, Remark \ref{rem:killingdrift} applies, and $T = 1$. For the bundle condition, a direct computation gives $(d\varphi_t)^{-1}(D)\big|_\zeta = (\mathrm{Ad}_{a_t^{-1}}\mathfrak{d})\cdot\zeta$, so Assumption \ref{ass:symmetry}(i) reads $[A_d, \mathfrak{d}] \subseteq \mathfrak{d}$. This holds whenever $H_d$ is 1-local: if $A$ is a Pauli string with $\mathrm{supp}(A) = \{i\}$ and $Q$ is a Pauli string with $|\mathrm{supp}(Q)| \leq 2$, then $[A,Q]$ vanishes when $i \notin \mathrm{supp}(Q)$ and is supported in $\mathrm{supp}(Q)$ otherwise; in either case $[A,Q]$ is again 2-local, and the claim follows by bilinearity.

It is worth emphasising that Assumption \ref{ass:symmetry}(i) requires the drift to \emph{normalise} $\mathfrak{d}$, which is a strictly stronger condition than membership $A_d \in \mathfrak{d}$. Indeed, a 2-local drift generally fails it: taking $n = 3$ and writing $X = \sigma_1$, $Y = \sigma_2$, $Z = \sigma_3$, both $-i\,Z \otimes Z \otimes I$ and $-i\,I \otimes X \otimes X$ lie in $\mathfrak{d}$, whereas
\begin{equation}
    \bigl[\,-i\,Z \otimes Z \otimes I,\; -i\,I \otimes X \otimes X\,\bigr] \;=\; Z \otimes [Z,X] \otimes X \;=\; 2\bigl(-i\,Z \otimes Y \otimes X\bigr)
\end{equation}
has support $\{1,2,3\}$ and is therefore 3-local. This is precisely the mechanism that makes $\mathfrak{d}$ bracket-generating, so the failure is not incidental; for such drifts the reduction of this section does not apply, and we return to them in Subsection \ref{sec:generalcase}. 

Minimising the control energy subject to locality constraints on the Hamiltonian is thus a prototypical example of the type of control problem studied in this section. Indeed, the work presented in this paper was motivated initially by studying this quantum control problem using the techniques of (sub)-Riemannian geometry.

We now study the convergence of the solutions of the soft-constrained problem to solutions of the hard-constrained problem in this setting. Let $P_{\mathfrak{d}}$ and $P_{\mathfrak{d}^\perp}$ denote the $\langle\cdot,\cdot\rangle$-orthogonal projections onto $\mathfrak{d}$ and $\mathfrak{d}^\perp$. The hard-constrained problem \eqref{eq:hardcontrol} asks for controls $A_c(t) \in \mathfrak{d}$ of minimal energy $\int_0^1 \langle A_c, A_c\rangle\,dt$ steering $x$ to $y$; the soft-constrained problems \eqref{eq:softcontrol} instead admit arbitrary $A_c(t) \in \mathfrak{g}$ and penalise the non-2-local component,
\begin{equation}
    \int_0^1 \Bigl( \bigl\langle P_{\mathfrak{d}}A_c,\, P_{\mathfrak{d}}A_c\bigr\rangle + q\,\bigl\langle P_{\mathfrak{d}^\perp}A_c,\, P_{\mathfrak{d}^\perp}A_c\bigr\rangle \Bigr)\,dt,
\end{equation}
which is precisely the cost induced by the penalised metric \eqref{eq:gq} under right invariance. By Proposition \ref{prop:geodesiccorrespondence} and Theorem \ref{thm:maincontroltheorem} (with $T = 1$), the soft problem at parameter $q$ is solved by a minimising geodesic $\xi_q$ of $g_q$ from $x$ to the rotated target $y' = \varphi_1^{-1}(y) = e^{iH_d}y$, with physical trajectory $\gamma_q(t) = e^{-iH_dt}\xi_q(t)$ and control recovered from $Y = \dot{\gamma} - X$. As $q \to \infty$, the $\xi_q$ subconverge strongly in $H^1$ to a sub-Riemannian minimiser, the controls converge in $L^2$ along the embedding by Theorem \ref{thm:maincontroltheorem}(iv), and the optimal energies increase to $d_\infty(x, e^{iH_d}y)^2$. Moreover, the guarantees of Subsection \ref{sec:subsequenceselection} apply to the full sequence: whether or not a distinguished subsequence has been selected, the non-2-local component of the optimal control vanishes in $L^2$ at the explicit rate $O(q^{-1/2})$ (Proposition \ref{prop:aprioribound}), and indeed $q\int_0^1 |P_{\mathfrak{d}^\perp}A_c^q|^2\,dt \to 0$ (Corollary \ref{cor:fullnormal}).

In the absence of drift, this soft-constrained problem has been studied extensively. The penalised inner products above are precisely the \emph{penalty metrics} of the complexity-geometry programme \citep{nielsen2006geometric, dowling2006}, introduced so that lengths of $g_q$-geodesics capture the difficulty of synthesising a unitary from local interactions, and the hard-constrained limit is the sub-Riemannian (Carnot--Carath\'{e}odory) geometry of the 2-local distribution. The limit $q \to \infty$ in this setting was investigated by Shizume et al.\ \cite{shizume}, by different methods. The results of Section \ref{sec:convergence} strengthen such statements from convergence of distances to strong $H^1$-convergence of the minimising trajectories themselves, together with quantitative control of the constraint violation and $L^2$-convergence of the controls. Theorem \ref{thm:maincontroltheorem} extends the programme to systems with a 1-local drift term. We note that this example instantiates the invariant sub-Riemannian geometry induced on a compact group by its Killing form, which has been studied in \cite{BoscainRossi2008}. 

\subsection{Arbitrary (non-symmetric) drift}\label{sec:generalcase}

We conclude this work by sketching how the situation changes when the symmetry condition (Assumption \ref{ass:symmetry}) is dropped. This symmetry assumption is somewhat restrictive: as we saw in Subsection \ref{sec:quantumexample}, it arises on $SU(2^n)$ as soon as the drift is 2-local. We will defer the full proof to a subsequent work; in this section we content ourselves with the rough details. 

The natural way to autonomise \eqref{eq:Ytildeopt} is to promote time to a coordinate. On the augmented manifold $M \times \R$, with points $(p,s)$, define the metric and subbundle
\begin{equation}
    h\big|_{(p,s)} := \tilde{g}_s\big|_p \oplus ds^2, \qquad \mathcal{E}\big|_{(p,s)} := (\tilde{D}_s)\big|_p \oplus \R\,\partial_s.
\end{equation}
For $q \in \Npos$, the energies $E^h_q$ associated to the penalisation of $h$ with respect to $\mathcal{E}$, act on curves $\eta = (\zeta, s) : [0,1] \to M\times\R$. Note that $\mathcal{E}$ is bracket-generating whenever each $\tilde{D}_s$ is, and that the $s$-direction is never penalised. A curve of the form $\eta(t) = (\zeta(t), t)$ satisfies
\begin{equation}
    2E^h_q(\eta) - 1 \;=\; \int_0^1 (\tilde{g}_t)_q(\dot{\zeta},\dot{\zeta})\,dt,
\end{equation}
and $\eta$ is $\mathcal{E}$-horizontal if and only if $\dot{\zeta}(t) \in (\tilde{D}_t)_{\zeta(t)}$ for almost every $t$. The soft and hard problems are therefore equivalent to the minimisation of $E^h_q$ and $E^h_\infty$ over $\mathcal{A}$, where $\mathcal{A}$ denotes the class of curves of the form $\eta(t) = (\zeta(t), t)$ and $\zeta(0) = x$ and $\zeta(1) = y' := \varphi_1^{-1}(y)$.

The crucial point is that, in the absence of Assumption \ref{ass:symmetry}, these problems are \emph{not} equivalent to the unconstrained geodesic problem on $(M\times\R, h)$. Since $h$ depends on the coordinate $s$ whenever the drift is not symmetric, an unconstrained minimiser is free to traverse the $s$-direction non-uniformly. In particular, unconstrained geodesics can remain constant on $M$, traversing in the $s$ direction to a value of $s_1$ such that $\tilde{g}_s$ is small, then traverse from $(x, s_1)$ to $(\varphi_1^{-1}(y), s_1)$ on the slice $M\times \{s_1\}$ whose diameter (with respect to $h$) is small, and then move again via the constant (on $M$) path to $(\varphi_1^{-1}(y), 1)$.

Its $s$-component need not be affine, and hence its projection to $M$ is then meaningless as a control trajectory, the $s$-coordinate being time. Even simple examples, such as a dilational drift on $\R$, exhibit a strict gap between the unconstrained minimum and the minimum over $\mathcal{A}$. The constraint $s(t) = t$ cannot be avoided.

Let us sketch out how the convergence theory of Section \ref{sec:convergence} can be adapted to this more general setting.
\begin{itemize}
    \item Two hypotheses are needed: (a) $(M\times\R, h)$ is complete; (b) the hard-constrained problem is feasible with finite cost. These replace the completeness of $(M,g)$ and the Chow--Rashevskii theorem guaranteeing the existence of a finite-energy competitor, neither of which is available over the constrained class. Both hold automatically when $X = 0$, and (a) holds whenever $M$ is compact, or when $c^{-1}g \leq \tilde{g}_t \leq c\,g$ uniformly in $t$.
    \item The $\Gamma$-convergence of the extended energies in the uniform topology (Theorem \ref{thm:gammaC0}) is expected to carry over verbatim, with $\mathcal{C}_{x,y}$ replaced by $\mathcal{A}$ throughout.
    \item Under (b), minimising families satisfy a uniform energy bound, and the Arzel\`{a}--Ascoli argument of Proposition \ref{thm:limitconstruct} produces a uniformly convergent subsequence; since $\mathcal{A}$ is closed under uniform convergence, the limit remains in $\mathcal{A}$.
    \item The most significant change is that the existence of minimisers of the penalised problems no longer follows from the Hopf--Rinow theorem. It must instead be established by the direct method over $\mathcal{A}$, using (a).
    \item The upgrade to strong $H^1$ convergence, via energy convergence and the Radon--Riesz property, is unaffected.
    \item The conclusion takes the same form as Theorem \ref{thm:maincontroltheorem}: penalised minimisers exist, subconverge strongly in $H^1$ to a solution of the hard-constrained problem, and the optimal values converge. The limiting object, however, is a constrained minimiser on the augmented manifold, rather than a geodesic of any fixed metric on $M$.
\end{itemize}

A more natural approach is to work directly with the time-dependent Lagrangians $L^q_t(\zeta,\dot{\zeta}) = \frac{1}{2}(\tilde{g}_t)_q(\dot{\zeta},\dot{\zeta})$, rather than lifting to $M \times \R$. The methods of Section \ref{sec:convergence} adapt to this setting with minor changes: the proofs there use only the joint continuity and local uniform ellipticity of the metric coefficients, which the time-dependent coefficients of $(\tilde{g}_t)_q$ inherit from Assumption \ref{ass:flow}. In particular, the associated action functionals $\Gamma$-converge, with respect to uniform convergence, to the constrained action $\frac{1}{2}\int_0^1 \tilde{g}_t(\dot{\zeta},\dot{\zeta})\,dt$ on curves with $\dot{\zeta}(t) \in (\tilde{D}_t)_{\zeta(t)}$ for almost every $t$. 

What is lost is reparametrisation invariance, and with it the constant-speed and distance statements of Theorem \ref{thm:penalisedconvergence}. Adapting these results to the time-dependent case will be the subject of future work. 
\section*{Acknowledgments}

The authors would like to thank Andrew Lewis, Richard Murray, Richard Montgomery and Andrei Agrachev for their comments. This work was partly funded by the European Union under the project ROBOPROX (reg.\ no.\ CZ.02.01.01/00/22\_008/0004590).

\section*{AI tool use declaration}

While the preprint that was submitted to arXiv on the 24th of April was almost entirely human in nature, in producing this revision the authors have made more substantial use of AI tools. In accordance with the recommendations of the Leiden Declaration on Artificial Intelligence and Mathematics\footnote{\url{https://leidendeclaration.ai/}}, the authors feel it is important to declare the use of AI tools, and to state that all proofs have been independently verified by the authors. The authors accept full responsibility for the content of the publication. 

The tools used were the following: OpenAI's GPT-5 was used for literature search, Anthropic's Claude Sonnet 4.6 was used for the formalisation of the ``manifold interaction picture'' and minor copyediting. OpenAI's GPT-5.5 discovered two non-trivial errors in the proof. Claude Fable 5 was used to patch these errors, culminating in the present work, and ChatGPT-5.6 was used for final proofreading. 

\bibliographystyle{alpha}
\bibliography{refs}

\end{document}